\pdfoutput=1
\documentclass[11pt, a4paper, ngerman, UKenglish]{scrartcl}

\usepackage[utf8]{inputenc}
\usepackage[ngerman, main=UKenglish]{babel}
\usepackage[pdftex,hyperfootnotes=false,pdfpagelabels, hidelinks,psdextra]{hyperref}
\usepackage[dvipsnames]{xcolor}
\usepackage{graphicx}

\usepackage{enumerate}
\usepackage{csquotes}
\usepackage{comment}
\usepackage{xspace}

\let\oldtexttt\texttt
\renewcommand{\texttt}[1]{{\small\oldtexttt{#1}}}

\usepackage{setspace}
\usepackage{relsize}

\newlength{\alphabet}
\usepackage{geometry}
\usepackage[backend=biber, style=alphabetic, maxbibnames=5]{biblatex}
\usepackage{amsfonts}
\usepackage{libertinus}
\usepackage{libertinust1math}
\usepackage[T1]{fontenc}
\usepackage{eucal}
\usepackage{beramono}
\usepackage[semibold]{sourcesanspro}
\usepackage{mathrsfs}
\usepackage{bm}
\usepackage{enumitem}

\usepackage{textcase}
\newcommand{\flatcaps}[1]{\large{\textsc{\MakeTextLowercase{#1}}}}

\usepackage{ragged2e}
\usepackage[protrusion=true, expansion=true]{microtype}

\nobreak

\PassOptionsToPackage{norule, ragged, hang}{footmisc}
\RequirePackage{fnpct}
\RequirePackage{footmisc}

\makeatletter
\renewcommand\@makefntext[1]{
    \@thefnmark.~#1}
\makeatother

\addtokomafont{sectioning}{\sffamily}

\usepackage{booktabs}
\usepackage{multirow}
\usepackage{longtable}
\usepackage{rotating}

\usepackage{wrapfig}
\usepackage{graphicx}
\usepackage[font=small, labelfont={sf, bf}]{caption}
\usepackage{subcaption}

\usepackage{tikz-cd}
\usetikzlibrary{decorations.pathmorphing, positioning}

\usepackage{mathtools}
    \mathtoolsset{showonlyrefs}
\usepackage{amsthm}
\usepackage{amsmath}
\usepackage{amssymb}
\usepackage{tensor}
\usepackage{mathdots}
\usepackage{nicefrac}
\usepackage{xfrac}
\usepackage[lined, algoruled, english]{algorithm2e}
    \SetAlgoCaptionSeparator{.}
    
\numberwithin{equation}{section}

\usepackage{lipsum}
\usepackage{pgfgantt}

\newtheoremstyle{TEMPLATE}% name
{} %       Space above, empty = 'usual value'
{} %       Space below
{} %       Body font
{} %       Indent amount (empty = no indent, \parindent = para indent)
{} %       Thm head font
{} %       Punctuation after thm head
{} %       Space after thm head: \newline = linebreak
{} %       Thm head spec

\newtheoremstyle{majorStatement}
    {3\topsep}
    {3\topsep}
    {\itshape}
    {}
    {\bfseries}
    {}
    {\newline} 
    {\sffamily\bfseries\thmname{#1} \thmnumber{#2}
     \nopagebreak \medskip \nopagebreak}

\newtheoremstyle{minorStatement}
    {}
    {}
    {\itshape}
    {0pt}
    {\bfseries}
    {.~}
    {0pt} 
    {\sffamily\thmname{#1} \thmnumber{#2}\thmnote{\normalfont (#3)}}
    
\theoremstyle{minorStatement}
\newtheorem{proposition}{Proposition}[section]
\newtheorem{lemma}[proposition]{Lemma}
\newtheorem{corollary}[proposition]{Corollary}

\theoremstyle{majorStatement}
\newtheorem{theorem}[proposition]{Theorem}

\newtheoremstyle{titledDefinition}
    {2\topsep}
    {2\topsep}
    {\normalfont}
    {0pt}
    {\bfseries}
    {}
    {\newline} 
    {\smallskip\sffamily\thmname{#1} \thmnumber{#2}\thmnote{\normalfont\itshape ~--~#3~--}}

\theoremstyle{titledDefinition}
\newtheorem{definition}[proposition]{Definition}

\newtheoremstyle{titledTheorem}
    {3\topsep}
    {5\topsep}
    {\itshape}
    {}
    {\bfseries}
    {}
    {\newline} 
    {\thmnote{\normalfont\flatcaps{#3}}
     \nopagebreak \medskip \nopagebreak}

\theoremstyle{titledTheorem}

\newtheoremstyle{otherStatement}
    {\topsep}
    {\topsep}
    {\normalfont}
    {0pt}
    {\bfseries}
    {.~}
    {0pt}
    {\sffamily\thmname{#1} \thmnumber{#2}}

\theoremstyle{otherStatement}

\newtheorem{remark}[proposition]{Remark}

\newtheorem{condition}[proposition]{Condition}

\renewenvironment{proof}
    {\par\vspace{\topsep}{\sffamily\bfseries Proof.~}}{\nopagebreak\par\hfill\qed\linebreak}
\newcommand{\CHAIN}{DART}

\newcommand{\R}{\mathbb{R}}
\newcommand{\N}{\mathbb{N}}

\newcommand{\seq}[2]{\left( #1 \right)_{ #2 \in \mathbb{N}}}
\newcommand{\cmpl}[1]{{#1}^{\mathrm{c}}}

\newcommand{\mc}[1]{\mathcal{ #1 }}
\newcommand{\mf}[1]{\mathfrak{ #1 }}

\newcommand{\mbb}[1]{\mathbb{ #1 }}
\newcommand{\mrm}[1]{\mathrm{ #1 }}

\DeclareMathOperator{\diag}{\mathrm{diag}}

\renewcommand{\d}[1]{~\mathrm{d}{#1}}
\newcommand{\intdom}[1]{\int_{#1}\mspace{-4mu}}
\newcommand{\argmin}[1]{\mathrm{arg}\min_{ #1 }}

\newcommand{\norm}[2][]{%
  \ifx#1\relax
    \| \mspace{1mu} {#2} \mspace{1mu} \|
  \else
        #1\| \mspace{1mu} {#2} \mspace{1mu} #1\|
  \fi
}
\newcommand{\enorm}[2][]{%
  \ifx#1\relax
    \| \mspace{1mu} {#2} \mspace{1mu} \|_{2}
  \else
        #1\| \mspace{1mu} {#2} \mspace{1mu} #1\|_{2}
  \fi
}

\newcommand{\tvnorm}[2][]{%
  \ifx#1\relax
    \big\| \mspace{1mu} {#2} \mspace{1mu} \big\|_{\tv}
  \else
        #1\| \mspace{1mu} {#2} \mspace{1mu} #1\|_{\tv}
  \fi
}

\newcommand{\iprod}[2][]{%
  \ifx#1\relax
    \mspace{2mu} \langle {#2} \rangle 
  \else
        \mspace{2mu} #1\langle {#2}  #1\rangle
  \fi
}

\DeclareMathOperator{\pr}{\mrm{Pr}}
\DeclareMathOperator{\exv}{\mathbb{E}\mspace{-1mu}}

\newcommand{\given}{~\big|~}
\newcommand{\tv}{\mathsmaller{\mathrm{TV}}}
\newcommand{\klDiv}[2]{\mrm{KL}\big({#1} \, \|~{#2}\big)}

\newcommand{\ideal}[1]{\hat{#1}}

\usepackage{float}
\usepackage{placeins}
\makeatletter
\def\fps@figure{htbp}
\makeatother

\usepackage{etoolbox}
\preto{\section}{\FloatBarrier}
\preto{\subsection}{\FloatBarrier}

\title{Fast-Mixing Markov Chains without Gradients}
\author{%
  Robert~Kutri\thanks{Institute for Mathematics and Interdisciplinary Center for Scientific Computing (IWR), Heidelberg University, 69120 Heidelberg, Germany (robert.kutri@uni-heidelberg.de).}%
  ,~Robert~Scheichl\thanks{Institute for Mathematics and Interdisciplinary Center for Scientific Computing (IWR), Heidelberg University, 69120 Heidelberg, Germany (r.scheichl@uni-heidelberg.de).}%
}
\date{\today}

\begin{document}

    \maketitle

    Labels: \textit{Markov Chain Monte Carlo, SGLMM, Gaussian Process, Surrogates}

    \begin{abstract}
    \noindent
Most approaches for accelerating Markov chain mixing either rely
on incorporating expensive geometric information in the proposals,
or reduce the per-step cost of sampling via surrogate densities.
We propose a localisation principle that allows a surrogate-based
Metropolis--Hastings proposal to exploit gradient-level geometric
information about the target density, without evaluating either the
target gradient or the surrogate gradient. The construction relies on
regularisation and tempering of the proposal measure. We show that the
expected proposal displacement coincides with the Langevin drift up to
controlled error. The resulting framework, Delayed Acceptance with
Regularisation and Tempering (DART), achieves an
\(\mathcal{O}(\kappa \max\{\kappa, d\})\) mixing time from warm start
for strongly log-concave targets with condition number~\(\kappa\)
in~\(d\) dimensions. This matches the known \(\mathcal{O}(\kappa d)\) rate
for MALA when \(d \ge \kappa\), and scales as
\(\mathcal{O}(\kappa^2)\), independent of dimension, otherwise. This
is, to our knowledge, the first mixing time guarantee for a
surrogate-transition-based MCMC method.
We demonstrate DART on a hierarchical spatial generalised linear mixed model.
In this setting, the Dirichlet–Neumann averaging parametrisation, originally
introduced for the efficient simulation of Gaussian processes, is repurposed
to supply the surrogate, and its linear memory and log-linear arithmetic
scaling in the number of observation sites carry over to inference.

    \end{abstract}

    \pagebreak

    \section{Motivation}   
        \label{sec:motivation}
        We consider the problem of sampling from a probability measure \(\Pi\)
with density
\begin{equation} \label{eq:targetDensity}
  \pi(x) \propto e^{-f(x)}, \quad \text{for } x \in \R^d,
\end{equation}
where the continuous function \(f \colon \R^d \to \R\) arises from
a computationally intensive model and cannot be expressed in closed
form.

Problems of this kind arise throughout statistics and probabilistic
machine learning. The function \(f\) may represent the negative
log-posterior of a Bayesian model obtained through a costly likelihood
evaluation
    \cite{stuart2010inverse,conrad2016accelerating}, or
a potential arising from marginalising over high-dimensional latent
states
    \cite{diggle1998model, andrieu2010particle, andrieu2009pseudo}.
In each setting, pointwise evaluation of \(f\) at scale is prohibitively
expensive, making direct sampling from \(\Pi\) intractable.
In moderate-to-high dimension, or when \(f\) induces complicated
geometry, MCMC methods are the standard tool. They require only
pointwise evaluations of \(f\) or its gradient and yield
asymptotically exact samples from \(\Pi\).

Their practical efficiency depends on several interacting design choices.
The key factors we consider are the mixing time of the chain, which
quantifies the number of steps required to reach approximate stationarity,
the per-step cost, and the parametrisation of the target density.
A suitable parametrisation of \(\pi\) can alter the geometry
of the sampling target, affecting both mixing and per-step cost simultaneously.

Incorporating local geometric information about \(f\) into the
transition mechanism is one of the most effective ways to accelerate
mixing. Methods based on Langevin dynamics
    \cite{roberts1996langevin, dalalyan2017theoretical,
          bou2013nonasymptotic}
and Hamiltonian Monte Carlo
    \cite{duane1987hybrid, neal2011mcmc, hoffman2014no}
exploit gradient information to construct proposals that respect the
local curvature of \(f\), and rigorous mixing time guarantees are
available in several settings
    \cite{durmus2017convergence, eberle2014error, dwivedi2019log,
          chen2020fast}.
These gains, however, come at a price. Each transition requires one
or more evaluations of \(\nabla f\), with methods such as
stochastic-Newton MCMC \cite{martin2012stochastic} additionally
requiring Hessian evaluations. Moreover, when the target is
multimodal, gradient-driven dynamics tend to mix well within
individual modes but struggle to traverse the low-probability
regions separating them.

Turning to the second strategy, the most direct way to reduce
per-step cost is to replace \(f\) by a cheaper surrogate
\(g\) induced by a lower-fidelity model, such as a Gaussian
process
    \cite{gramacy2020surrogates, stuart2018posterior},
a neural operator
    \cite{bhattacharya2021model, lu2021learning,
          raonic2023convolutional},
or a polynomial approximation
    \cite{westermann2026performance}.
Substituting \(g\) for \(f\), however, yields an
invariant measure that is biased relative to \(\Pi\). Asymptotic
exactness can be retained by embedding surrogates more carefully into
the MCMC procedure. The surrogate transition method (STM)
\cite{liu2001monte} runs an auxiliary chain targeting a surrogate
measure \(\tilde\Pi\) with density
\(\tilde\pi \propto e^{-g}\) and corrects the resulting
proposals via a Metropolis--Hastings step against the true target
density. Delayed acceptance (DA) MCMC \cite{christen2005markov,cui2011bayesian}
uses \(g\) to pre-screen proposals before evaluating the expensive
model. Both approaches preserve \(\Pi\) as the invariant measure of
the overall chain.

The Multi-Level Delayed Acceptance (MLDA) method
\cite{lykkegaard2023multilevel} combines surrogate transitions and
screening across multiple surrogate levels. Explicit non-asymptotic
mixing time bounds in terms of the problem parameters, of the kind
available for MALA or HMC \cite{dwivedi2019log, chen2020fast}, are to
our knowledge not available for surrogate-transition, or delayed-acceptance
methods. The efficiency of these methods further depends critically on the
availability of a surrogate hierarchy whose geometry does not differ
too severely from that of the target, and this hierarchy is fixed a
priori. Adaptive error modelling schemes \cite{lykkegaard2020multilevel}
partially address this by adjusting the surrogate online, but come
without a priori guarantees.

        \subsection*{Our Contributions}
            This manuscript makes three contributions to the design and analysis of
surrogate-driven MCMC methods.

% Contribution 1: methodological
The first is methodological. We introduce a localisation principle, achieved
through a combination of regularisation and tempering of the surrogate density.
This allows a surrogate density with sufficient gradient fidelity to carry the
same geometric information as an exact gradient-based proposal, without requiring
evaluations of either \(\nabla f\) or \(\nabla g\). When the
surrogate \(g\) is quadratic, the localised surrogate density is Gaussian and
proposals can be drawn directly, reducing the algorithm to a
Metropolis--Hastings variant with a specific, surrogate-informed proposal. When
the surrogate is not quadratic, we draw approximate samples by running a short
auxiliary Markov chain targeting the localised surrogate density. This extends
the surrogate transition framework of \cite{liu2001monte} and the delayed
acceptance approaches of \cite{christen2005markov, lykkegaard2023multilevel}
by allowing for localised, state-dependent proposals. We refer to the
resulting method as Delayed Acceptance with Regularisation and Tempering (DART).

% Contribution 2: theoretical
The second is theoretical. We prove that for strongly log-concave targets with
condition number \(\kappa\), the DART chain achieves a total-variation mixing
time of \(\mathcal{O}(\kappa \max\{\kappa, d\})\) from a warm start, provided
the surrogate has sufficient gradient fidelity. This bound exhibits two regimes.
When \(d \geq \kappa\), it matches the \(\mathcal{O}(\kappa d)\) rate established
for MALA in \cite{dwivedi2019log}.%
\footnote{%
    This MALA bound has subsequently been sharpened to the minimax-optimal
    \(\tilde{\mathcal{O}}(\kappa\sqrt{d})\) \cite{wu2022minimax}.
    Since our argument closely parallels that of \cite{dwivedi2019log}, we
    expect the same refinements to transfer. We do not pursue this here.%
}
When \(\kappa \geq d\), it becomes \(\mathcal{O}(\kappa^2)\), with no explicit
dependence on \(d\). To our knowledge, it is the first explicit
mixing time bound for a surrogate-transition-based MCMC method. 

The third is of independent interest, but strong synergy with DART. The
Dirichlet--Neumann averaging (DNA) construction of
\cite{kutri2026dirichlet} was developed for the efficient simulation of
Gaussian processes on regular grids. We repurpose it as a
parametrisation of the latent field in Bayesian models with isotropic
Gaussian priors, where inference at scale is typically bottlenecked by
a factorisation of the prior and computation of its log-determinant,
as well as the simulation of corresponding realisations.
We demonstrate the combination on a
hierarchical spatial generalised linear mixed model (SGLMM).
Under this new parametrisation, and for any sampler, the per-iteration
latent field operations collapse from quadratic to near-linear:
\begin{center}
\begin{tabular}{lccc}
\toprule
                  & latent field evaluation  & factorisation & prior log-determinant\\
\midrule
Naive (Cholesky)   & \(\mathcal{O}(d^2)\) & \(\mathcal{O}(d^3)\)  & \(\mathcal{O}(d^3)\) \\
DNA                & \(\mathcal{O}(d \log d)\)   & \(\mathcal{O}(d)\) & \(\mathcal{O}(d)\) \\
\bottomrule
\end{tabular}
\end{center}
The parametrisation separates the latent field into low- and high-frequency components
to yield a natural surrogate structure for DART. Combining the DNA parametrisation
and DART unlocks log-linear computational scaling and linear memory scaling for
hierarchical SGLMM problems with respect to the number of observations.

The manuscript is organised as follows.
Section~\ref{sec:setting} specifies the general mathematical
setting, fixes notation, and recalls the necessary background on
MCMC methods and mixing times. Section~\ref{sec:dart} introduces
the DART framework, including the algorithmic construction, the
main theoretical results, and practical aspects of implementation.
Numerical experiments illustrating the theoretical findings are
given in Sections~\ref{sec:numerics}. Section~\ref{sec:sglmm}
introduces the novel Gaussian prior parametrisation and its 
synergy with DART. Finally,
Section~\ref{sec:proof} contains the proofs of the main results.

    \section{Preliminaries}  
        \label{sec:setting}
        \subsection{Notation}  
            \label{sec:setting:notation}
            We write \(\pr\) for the probability measure on a common underlying space
carrying all random variables, and \(\mbb{E}\) for the corresponding
expectation. All measures on \(\R^d\) are defined on the Borel
\(\sigma\)-algebra \(\mf{B}(\R^d)\). Throughout, \emph{measurable} means
Borel measurable. Absolute continuity and densities are understood with
respect to the Lebesgue measure, and we leave this implicit.

We write \(\mf{P}(\R^d)\) for the space of Borel probability measures on
\(\R^d\), denote the Gaussian measure with mean \(\bar{x}\) and covariance
\(C\) by \(\mc{N}(\bar{x}, C)\), and the Dirac measure at \(x\) by
\(\delta_x\).
For \(\nu \in \mf{P}(\R^d)\) and \(\psi \sim \nu\) we write
\(\mbb{E}_\nu[\cdot]\) for expectation under \(\nu\).

The total variation distance between \(\mc{M}_1, \mc{M}_2 \in \mf{P}(\R^d)\)
is
\begin{equation}    \label{eq:tvDistance}
    \tvnorm{\mc{M}_1 - \mc{M}_2}
        \coloneqq \sup_{A \in \mf{B}(\R^d)}
            \big|\, \mc{M}_1(A) - \mc{M}_2(A) \,\big|.
\end{equation}
With this normalisation, every \(\mc{M} \in \mf{P}(\R^d)\) satisfies the
functional bound
\begin{equation}    \label{eq:tvFunctionalBound}
    \left|\, \int f \,\d{\mc{M}_1} - \int f \,\d{\mc{M}_2} \,\right|
        \le \tvnorm{\mc{M}_1 - \mc{M}_2}
\end{equation}
for any measurable \(f \colon \R^d \to [0, 1]\).

We write \(\mbb{B}(x, R)\) for the closed Euclidean ball centred at \(x\)
with radius \(R\), and \(\enorm{\cdot}\) for the Euclidean norm. For
non-negative quantities, \(A \lesssim B\) means \(A \le c B\) for a
universal constant \(c\) that depends on none of the problem parameters,
\(A \gtrsim B\) means \(B \lesssim A\), and \(A \asymp B\) means both
\(A \lesssim B\) and \(B \lesssim A\). The same symbol may denote different
constants in different expressions. We write \(f \le g\) for pointwise
inequality between functions.

Given a density \(\pi\) on \(\R^d\), we write \(\pi \propto e^{-f}\) to mean
\(\pi(x) = N^{-1} e^{-f(x)}\), where \(N \coloneqq \int_{\R^d} e^{-f(x)}
\,\d{x}\) is the normalisation constant.

        \subsection{Markov Chains and Metropolis--Hastings}   
            \label{sec:setting:mcmc}
            A Markov chain on $\R^d$ is a sequence of random variables
\(
    \Psi = \seq{\psi_i}{i}
\)
whose dynamics, is fully determined by a family of transition
probability measures
\(
    \{ \mc{T}_x \}_{x \in \R^d}
\)
satisfying
\begin{equation}    \label{eq:transitionMeasureDef}
    \pr \big(\psi_i \in A \given \psi_{i-1} = x \big)
        = \mc{T}_{x} (A)
\end{equation}
for any \(A \in \mf{B}(\R^d)\) and \(i \in \N\).
The initial state follows a probability measure
\(\psi_0 \sim \mu\), referred to as the
\emph{initial measure}.
We assume throughout that \(x \mapsto \mc{T}_x(A)\) is
measurable.
The associated \emph{transition operator} is then
\begin{equation}
    T: \mc{P}(\R^d) \to \mc{P}(\R^d),
        \qquad
            \big(T \nu\big) (A) \coloneqq
            \intdom{\R^d} \mc{T}_x (A) \d{\nu(x)},
        \quad
            A \in \mf{B}(\R^d),
\end{equation}
so that \(\psi_i \sim T^i \mu\) for any \(i \in \N\).
A probability measure \(\nu\) is \emph{invariant} under
\(T\) if \(T\nu = \nu\). We refer to \(T\) itself as the (Markov) chain,
when the context is clear.
The goal of MCMC is to construct a chain for which
\begin{equation}    \label{eq:tvConvergence}
    \tvnorm{T^i \mu - \Pi} \to 0,
    \quad \text{as} \quad i \to \infty,
\end{equation}
so that for large \(i\), the states of the chain can be
used to perform inference with respect to \(\Pi\).
The Metropolis--Hastings (MH) algorithm
\cite{metropolis1953equation, hastings1970monte} provides
a general construction achieving this. Given a family of
absolutely continuous proposal measures
\(
    \{ \mc{P}_x \}_{x \in \R^d}
\)
with densities \(p_x\), a proposal
\(z \sim \mc{P}_x\) is accepted as the next state with
probability
\begin{equation}    \label{eq:mhAcceptance}
    \alpha(x, z)
        \coloneqq
        \min \left\{
            1, \frac{\pi(z)\, p_{z}(x)}{\pi(x)\, p_{x}(z)}
        \right\},
\end{equation}
and otherwise the chain remains at \(x\).
Provided that \(\pi\) and the \(p_x\) are positive a.e.,
the resulting chain has invariant measure \(\Pi\) and
satisfies \eqref{eq:tvConvergence} for \(\Pi\)-almost
every initial state
\cite{roberts1996geometric, meyn2008markov}.
Two standard instantiations are the Metropolised random
walk (MRW), which uses Gaussian proposals
\(
    \mc{P}_x = \mc{N}(x, h^2 \mbb{I}),
\)
and the Metropolis-adjusted Langevin algorithm (MALA)
\cite{roberts1996langevin},
for a step size parameter \(h > 0\),
which shifts proposals toward regions of higher density
via
\(
    \mc{P}_x = \mc{N}\big(
        x + \tfrac{h^2}{2}\, \nabla \log \pi(x),\, h^2 \mbb{I}
    \big).
\)

        \subsection{Mixing Times and Conductance}      
            \label{sec:setting:mixing}
            While constructing a convergent Markov chain in the sense of
\eqref{eq:tvConvergence} is possible for a broad class of
target measures, quantifying the speed of this convergence
is considerably more delicate. In particular, it depends
sensitively on the geometry of the density $\pi$. We
measure the speed of convergence by the number of steps
required for a prescribed total-variation tolerance.

\begin{definition}[Mixing Time]
    Let \(\delta > 0\).
    For a fixed initial measure $\mu$ on $(\R^d, \mf{B}(\R^d))$,
    the \emph{$\delta$-mixing time} \(t_{\delta} (\mu)\)
    of a Markov chain with transition operator $T$ is defined as
    \begin{equation}
        t_{\delta}(\mu) \coloneqq
            \inf \left\{
                n \in \N: \tvnorm{T^n \mu - \Pi} < \delta
            \right\}.
    \end{equation}
\end{definition}

Mixing can be slow in general. For multi-modal target
densities, local proposals struggle to cross
low-probability regions separating modes, leading to
poor global exploration. Even for unimodal targets,
convergence can be arbitrarily slow if the chain is
initialised from an unfavourable position
\cite{bandeira2023free}. This motivates restricting
attention to classes of initial measures that represent
beneficial initial configurations.

\begin{definition}[Warmness]    \label{def:warm}
    Let $\nu$ be a probability measure on 
    \((\R^d, \mf{B}(\R^d))\) with
    continuous density $\varphi$.
    For $\beta \ge 1$, we say that
    $\nu$ is \emph{$\beta$-warm} with respect to $\Pi$
    if
    \begin{equation} 
        \sup_{x \in \R^d} \frac{\varphi(x)}{\pi(x)} \le \beta.
    \end{equation}
    We say that $\nu$ is \emph{warm} with respect to $\Pi$
    if it is $\beta$-warm for some finite $\beta$.
\end{definition}

A complementary notion is the \emph{conductance} $\Phi$
of a chain $T$ with invariant measure $\Pi$, defined by
\begin{equation}
    \Phi(T)
        \coloneqq \inf\limits_{\Pi(A) \in (0, \nicefrac{1}{2})}
            \frac{\intdom{A} \mc{T}_x(\cmpl{A})
                \, \pi(x) \d{x}}{\Pi(A)}.
\end{equation}
Conductance measures the minimal probability flux leaving
a set relative to its target mass and thus quantifies the
presence of \enquote{bottleneck regions} in the state space.
Cheeger constants in differential geometry and graph theory
are closely related notions.

To discount sets of negligible target mass, one introduces
the $s$-conductance $\Phi_s(T)$, defined for
$s \in (0, \tfrac{1}{2})$ by
\begin{equation}    \label{eq:sConductance}
    \Phi_s(T)
        \coloneqq
            \inf\limits_{
                \Pi(A) \in (s, \, \nicefrac{1}{2})
            }
            \frac{
                \intdom{A}~\mc{T}_x (\cmpl{A}) \, \pi(x) \d{x}
            }
            {
                \Pi(A) - s
            }.
\end{equation}

Warmness interacts favourably with $s$-conductance. Indeed, for any measurable $A \subset \R^d$ with $\Pi(A) < s$ and $s \in (0, \tfrac{1}{2})$, a $\beta$-warm measure $\nu$ satisfies
\begin{equation}    \label{eq:warmnessContinuityArgument}
    \nu(A)
        = \intdom{A} \varphi(x) \d{x} 
        \le \beta \intdom{A} \pi(x) \d{x}
        = \beta \Pi(A)
        \le \beta s,
\end{equation}
and therefore
\(
    |\nu(A) - \Pi(A)| \le \beta s.
\)

As a consequence, Theorem~1.4 and Corollary~1.6 in
\cite{lovasz1993random} imply that for a $\beta$-warm
initial measure $\mu$,
\begin{equation}    \label{eq:lovaszBound}
    \tvnorm{ T^i \mu - \Pi}
        \le \beta s
            + \beta \left(
                1 - \tfrac{1}{2} \Phi_s^2(T)
            \right)^i,
    \qquad i \in \N.
\end{equation}
Thus, lower bounds on the $s$-conductance
\eqref{eq:sConductance} combined with a suitable 
choice of $s$, yield explicit non-asymptotic
bounds on mixing times. This strategy has been
successfully employed in, for example,
\cite{lovasz1993random, lovasz1999hit, dwivedi2019log, chen2020fast},
and forms the basis of the analysis in Section~\ref{sec:proof}.
    \section{Localised Surrogate Transitions}
        \label{sec:dart}
        Let \(g \colon \R^d \to \R\) be a computationally cheaper
approximation to the target potential~\(f\). We construct a
state-dependent proposal measure \(\Pi_x\) by two
modifications of the surrogate density
\(\tilde{\pi} \propto e^{-g}\). First, similar to a
restricted Gaussian oracle \cite{titsias2018auxiliary, lee2021structured}
for the surrogate,
we \emph{localise} around the current state \(x \in \R^d\)
by adding a quadratic penalty scaled by \(\gamma > 0\), confining proposals to a
neighbourhood of \(x\) where short steps are more likely to be
accepted. Second, in the spirit of simulated annealing
\cite{kirkpatrick1983optimization,
marinari1992simulated, geyer1995annealing}, we \emph{temper}
the surrogate contribution by a factor \(\theta \in (0, 1]\),
making the resulting density easier to sample from. Together,
these modifications yield the potential
\begin{equation}    \label{eq:localisedSTDensity}
    V_x \colon \R^d \to \R,
    \qquad
    y \mapsto \theta\, g(y)
            + \frac{\gamma}{2} \, \enorm{y - x}^2,
\end{equation}
and we define the corresponding proposal measure \(\Pi_x\) with density
\begin{equation}    \label{eq:localisedDensity}
    \pi_x(y) \propto e^{-V_x(y)},
    \qquad y \in \R^d.
\end{equation}
Our core claim for the construction above is the following:
a draw \(z \sim \Pi_x\) has the same expected drift towards
high-density regions of \(\pi\) as a Langevin proposal, despite
never evaluating a gradient.

To illustrate when and why this might be true, define
the mean displacement
\(\bar{u}_x \coloneqq \exv_{u \sim \Pi_x} \! u\) of the
proposal for a  fixed state \(x \in \R^d\) and assume for
the moment that \(\nabla g = \nabla f\).
If \(\pi_x\) has (say) sub-Gaussian tails, then
\(\exv_{u \sim \Pi_x} \! \nabla V_x(u) = 0\) by integration
by parts. Expanding
\(\nabla V_x(u) = \theta\,\nabla g(u) + \gamma(u - x)\)
and rearranging yields
\begin{equation}    \label{eq:steinIdentity}
    \bar{u}_x
        = x - \frac{\theta}{\gamma}
            \exv_{u \sim \Pi_x} \! \nabla g(u)
        = x - \frac{\theta}{\gamma}
            \exv_{u \sim \Pi_x} \! \nabla f(u).
\end{equation}
The proposal mean is shifted from \(x\) along the
\(\Pi_x\)-averaged target gradient, with \(\theta / \gamma\)
acting as an implicit step size. A MALA proposal with step
size \(h = \sqrt{2 \theta / \gamma}\) produces the same drift with
the pointwise gradient \(\nabla f(x)\) in place of
the average gradient \(\exv_{u \sim \Pi_x} \! \nabla f(u)\).

Two technical hurdles stand in the way of the heuristic 
\eqref{eq:steinIdentity} reliably producing gradient-informed
proposals. No reasonably cheap surrogate can be expected 
to satisfy \(\nabla g = \nabla f\), so the first hurdle lies in
quantifying when the surrogate density carries sufficient
gradient information.
Secondly, for the average gradient to be a useful proxy
for the pointwise target gradient, the concentration of
\(\Pi_x\), controlled chiefly through increasing \(\gamma\),
needs to be strong enough. 
We address both issues for strongly log-concave target densities in
Theorem~\ref{thm:mixingTime}, where the gradient fidelity condition
\eqref{eq:fidelityCondition} and the required level of localisation
\eqref{eq:localisationCondition} are made explicit.

In DART, we therefore propose drawing \(z \sim \Pi_x\) and accepting it
via an MH step against the target density \(\pi\). Denote by
\(
    N_x \coloneqq \int_{\R^d} e^{-V_x(y)} \d{y}
\)
the normalisation constant of \(\pi_x\). Since the
densities \(\pi_x\) and \(\pi_z\) are normalised by different
constants when \(x \neq z\), these constants enter the
standard MH ratio~\eqref{eq:mhAcceptance} explicitly.
Owing to the symmetry of the quadratic penalty, we have
\begin{equation}    \label{eq:potentialSymmetry}
    V_x(z) - V_z(x)
        = \theta \big(g(z) - g(x)\big),
\end{equation}
and expanding the proposal density ratio gives the acceptance
probability
\begin{equation}    \label{eq:dartAcceptance}
    \alpha(x, z)
        = \min \left\{
            1,\;
            \frac{\pi(z)}{\pi(x)}
            \,
            \frac{N_x}{N_z}
            \, e^{\theta \, (g(z) - g(x))}
        \right\}.
\end{equation}
When \(z \sim \Pi_x\), the resulting Markov chain on \(\R^d\)
is reversible with respect to \(\Pi\) and admits \(\Pi\) as
its invariant measure.

Drawing \(z \sim \Pi_x\) exactly is feasible only in special
cases, the quadratic-surrogate case discussed in
Section~\ref{sec:dart:practice} being the main one. In
general, we approximate samples from~\(\Pi_x\) by running a
Markov chain, which we refer to as the \emph{root chain}, for
\(n\) steps. That is, we replace the exact proposal
\(\mc{P}_x = \Pi_x\) by
\begin{equation}    \label{eq:stProposal}
    \mc{P}_x = Q_x^n \, \mu_x \approx \Pi_x,
\end{equation}
where \(Q_x\) denotes the transition operator of a reversible
Markov chain with invariant measure~\(\Pi_x\) and \(\mu_x\) is an
initial measure. 

Because the density of \(Q_x^n \mu_x\) is in general intractable, we cannot
form its exact Metropolis ratio and instead retain
\eqref{eq:dartAcceptance}, the ratio for exact proposals
from~\(\Pi_x\). The implemented chain is then not in detailed
balance with \(\Pi\), and its bias is governed by the root-chain mixing
via \(\tvnorm{\Pi_x - Q_x^n \mu_x}\). Both localisation and
tempering reduce this error directly, the quadratic penalty making
the localised surrogate well-conditioned and tempering flattening it
further. Condition~\ref{cnd:root} fixes the root-chain length \(n\) this
requires, which Theorem~\ref{thm:mixingTime} subsequently folds into the
mixing-time bound.

This use of a state-dependent root chain is the sense in which DART
extends the surrogate transition method (STM) \cite{liu2001monte},
delayed acceptance (DA) \cite{christen2005markov}, and the
Multi-Level Delayed Acceptance (MLDA) method
\cite{lykkegaard2023multilevel}, all of which propose from a fixed,
global surrogate. Two-level MLDA is recovered as
\(\gamma \to 0\) and \(\theta = 1\), a regime favourable only
when the surrogate is at once faithful and cheap. 
Viewed from the proximal-sampling side, the
root chain realises a restricted Gaussian oracle on the
surrogate. The Metropolis step corrects it exactly, and our analysis
(Section~\ref{sec:proof}) shows the error left by running it for
finite time can be made as small as desired at only logarithmic
extra cost.

        \subsection{A Mixing Time Guarantee for DART}  
            \label{sec:dart:main}
            The main theoretical result of this manuscript is an explicit
upper bound on the mixing time of a Markov chain using DART from a warm
start, under the assumption of a strongly log-concave target
density. The bound applies to the general setting where
proposals are generated by a root chain of finite length
targeting the localised surrogate measure~\(\Pi_x\).
Corollary~\ref{cor:exactSampling} below shows that the
idealised and implemented chains coincide when proposals can be
drawn from~\(\Pi_x\) directly, as is the case for quadratic
surrogates, and the bound then holds in a stronger, monotone form.

Log-concave targets arise naturally in Gaussian models, logistic
regression, and other members of the exponential family of
response functions. More broadly, posterior measures arising
from non-linear statistical inverse problems admit accurate
log-concave approximations in Wasserstein distance under
suitable conditions \cite[Theorem~5.1.3]{nickl2023bayesian}.
We summarise the assumptions on the target geometry in the
following condition.

\begin{condition}   \label{cnd:logConcave}
    The measure \(\Pi\) is absolutely continuous with respect to the
    Lebesgue measure with density \(\pi \propto e^{-f}\), where
    \begin{enumerate}
        \item \label{cnd:logConcave:smooth}
            The potential \(f\) is continuously differentiable and has a
            Lipschitz continuous gradient with constant \(L > 0\). That
            is, for all \(x, y \in \R^d\),
            \begin{equation}
                \enorm{\nabla f(x) - \nabla f(y)}
                    \le L\,\enorm{x - y}.
            \end{equation}
        \item \label{cnd:logConcave:convex}
            The potential \(f\) is \(\lambda\)-strongly convex for some
            \(\lambda > 0\), meaning that for all \(x, y \in \R^d\),
            \begin{equation}
                \iprod{\nabla f(x) - \nabla f(y),\, x - y}
                \ge \lambda \,\enorm{x - y}^2.
            \end{equation}
    \end{enumerate}
\end{condition}

Strong convexity guarantees concentration of \(\Pi\) around a
unique mode \(x^{\star} \coloneqq \argmin{\R^d} f\)
(see e.g.~\cite{durmus2019high}). Lipschitz continuity of the
gradient provides quantitative control over how \(f\) varies
along a proposal move. We quantify the geometric complexity
of a target density satisfying Condition~\ref{cnd:logConcave}
through its \emph{condition number}
\begin{equation}    \label{eq:conditionNumber}
    \kappa \coloneqq \frac{L}{\lambda}.
\end{equation}

The following definition summarises the surrogate construction
for a target measure satisfying Condition~\ref{cnd:logConcave}.
\begin{definition}[Strongly Log-Concave Surrogate Density]
    \label{def:stronglyLogConcave}
    Let \(g : \R^d \to \R\) be \(m\)-strongly convex
    and \(M\)-smooth with unique minimiser
    \(\hat{x} \coloneqq \argmin{\R^d} g\).
    For any \(x \in \R^d\) and parameters
    \(\theta \in (0, 1]\), \(\gamma > 0\),
    define the potential
    \begin{equation}    \label{eq:surrogatePotentialDef}
        V_x(y) \coloneqq \theta \, g(y)
                    + \frac{\gamma}{2} \, \enorm{x-y}^2.
    \end{equation}
    Let \(\Pi_x\) denote the probability measure with density
    \(\pi_x\) proportional to \(e^{-V_x}\). We refer
    to \(\pi_x\) as a \emph{strongly log-concave surrogate
    density}.
\end{definition}

The analysis requires two kinds of agreement between \(g\) and
\(f\). The surrogate's mode \(\hat x\) must lie in the
same high-probability region of \(\Pi\) as the target's mode
\(x^{\star}\), formalised in Theorem~\ref{thm:mixingTime} by the
inclusion \(\hat x \in \mc{K}\). Within this region of high probability,
the gradient \(\nabla g\) must approximate \(\nabla f\) sufficiently
well. To ensure local regularity of \(\nabla g\) as \(\enorm{x-y} \to 0\),
we further require that \(g\) is itself \(m\)-strongly convex with an
\(M\)-Lipschitz gradient. Compatibility between these two requirements
is ensured by \(\lambda \le m \le M \le L\).

The final condition characterises the efficiency of the
root chain used by the algorithm. We assume the root chain
mixes in polynomial time in the dimension and condition number,
with exponents \(\omega, \tilde\omega > 0\). This is satisfied,
for instance, by MRW, MALA, and HMC targeting strongly log-concave
densities (see, for example, \cite{dwivedi2019log, chen2020fast}).
\begin{condition}   \label{cnd:root}
    Let \(\{Q_x\}_{x \in \R^d}\) be a family of transition
    operators where each \(Q_x\) has invariant measure \(\Pi_x\)
    with potential \(V_x\) that is strongly convex with Lipschitz
    continuous gradient. Denote by
    \(
        \tilde\kappa \coloneqq \sup_{x \in \R^d} \kappa(\Pi_x)
        \ge 1
    \)
    the largest condition number \eqref{eq:conditionNumber}
    among the \(\Pi_x\), assumed finite.
    There exist \(\omega, \tilde\omega > 0\) such that for every
    \(x \in \R^d\), every \(\beta_x\)-warm initial measure \(\mu_x\)
    with respect to \(\Pi_x\), and every \(\varepsilon \in (0, 1)\),
    \begin{equation}
        \tvnorm{Q_x^{n} \mu_x - \Pi_x} \le \varepsilon
        \quad \text{provided that} \quad
        n \gtrsim d^{\omega}
            \tilde\kappa^{\tilde\omega}
            \log \frac{2 \beta_x}{\varepsilon}.
    \end{equation}
\end{condition}
Under the parameter choices of Theorem~\ref{thm:mixingTime}, the
surrogate condition number \(\tilde\kappa\) is bounded by a universal constant
(see~\eqref{eq:surrogateConditionNumber}), and
Lemma~\ref{lem:surrogateWarmness} shows that the warmness requirement
is met by suitably concentrated Gaussian initial measures, with
\(\log \beta_x\) controlled uniformly over the region the chain
explores.

\begin{theorem} \label{thm:mixingTime}
    Let
    \(
        f: \R^d \to \R
    \)
    satisfy Condition~\ref{cnd:logConcave}
    and denote the mode
    \(x^{\star} \coloneqq \argmin{\R^d} f\).
    Let
    \(
        g: \R^d \to \R
    \)
    be an \(m\)-strongly convex surrogate for \(f\) with
    \(M\)-Lipschitz gradient, satisfying
    \(
        \lambda \le m \le M \le L,
    \)
    and denote
    \(
        \hat{x} = \argmin{\R^d} g.
    \)
    Let \(\Pi\) be the target measure with density
    \(
        \pi \propto e^{-f},
    \)
    let \(\mu\) be a \(\beta\)-warm start w.r.t.\
    \(\Pi\), and fix an error tolerance \(\delta \in (0,1)\)
    satisfying
    \(
        \log\frac{2\beta}{\delta} \le d.
    \)
    Define the ball
    \(
        \mc{K}
            \coloneqq \mbb{B}\bigl(
                x^\star,\,3\sqrt{d / \lambda}
            \bigr),
    \)
    and assume
    \(
        \hat x\in\mc{K}.
    \)
    Suppose the following hold:
    \begin{enumerate}[label=(\roman*)]
        \item
            \emph{Gradient fidelity.}
            \begin{equation}
                \label{eq:fidelityCondition}
                \sup_{x \in \mc{K}} \enorm{
                    \nabla f(x) - \nabla g(x)
                }^2
                    \lesssim L \max\{ 1,~\kappa / d\}.
            \end{equation}
        \item
            \emph{Localisation.} The DART parameters satisfy
            \begin{equation}
                \label{eq:localisationCondition}
                \theta=\frac{1}{2}
            \quad \text{and} \quad
                \gamma \simeq \, L \max\{ \kappa,~d\}.
            \end{equation}
        \item \emph{Root chain mixing.}
            The root chains
            \(
                \{Q_x\}_{x \in \R^d}
            \)
            satisfy Condition~\ref{cnd:root}, each initialised at
            the Gaussian measure
            \(
                \mu_x = \mc{N}\big(x,\,(2(\gamma + \theta M))^{-1}\mbb{I}\big)
            \)
            of Lemma~\ref{lem:surrogateWarmness}, and the root chain
            lengths satisfy
            \begin{equation}
                \label{eq:surrogateStepsCondition}
                n \gtrsim d^{\omega}
                    \Big(d + \log\tfrac{\kappa}{\delta}\Big).
            \end{equation}
    \end{enumerate}
    Then the \(\delta\)-mixing time of the DART chain satisfies
    \begin{equation}    \label{eq:mixingTimeBound}
        t_{\delta}(\mu)
            \lesssim \kappa \max\{ \kappa,~d\}
            \,\log\!\frac{2\beta}{\delta}.
    \end{equation}
\end{theorem}

When proposals are drawn from \(\Pi_x\) exactly, the perturbation
term \(\tvnorm{\Pi_x - Q_x^n \mu_x}\) vanishes, the chain satisfies
detailed balance with respect to \(\Pi\), and the mixing time bound
follows from conditions~(i) and~(ii) alone.
\begin{corollary}   \label{cor:exactSampling}
    Under the hypotheses of Theorem~\ref{thm:mixingTime},
    suppose that conditions~\textnormal{(i)} and
    \textnormal{(ii)} hold, and that proposals are drawn
    exactly from~\(\Pi_x\). Then
    Condition~\ref{cnd:root} and
    condition~\textnormal{(iii)} are not required, and
    the mixing time bound~\eqref{eq:mixingTimeBound} holds.
\end{corollary}
\begin{remark}
    \label{rmk:cosmetic}
    The condition \(\log(2\beta/\delta) \le d\) is cosmetic. 
    Removing it introduces a dependence on \(r(\delta/(4\beta), d)\)
    (see \cite[Theorems~1~and~2]{dwivedi2019log} and \eqref{eq:uniformRBound}),
    which increase only very slowly with decreasing error tolerance.
    As stated, Theorem~\ref{thm:mixingTime} restricts the
    mixing times to error tolerances \(\delta \gtrsim e^{-d}\).
    Furthermore, our results can be extended to weakly log-concave
    and perturbed log-concave targets using the reduction employed in
    \cite{dwivedi2019log}, ensuring algebraic mixing time bounds
    for those targets too. Since the arguments translate
    verbatim, we omit these computations here.
\end{remark}

\begin{remark}
    \label{rmk:twoRegimes}
    The mixing time bound \eqref{eq:mixingTimeBound} reveals two
    distinct regimes. If \(\kappa \ge d\), then
    \(t_{\delta}(\mu) \in \mc{O}(\kappa^2)\), with no explicit
    dependence on the dimension \(d\). If \(d \ge \kappa\), the
    bound becomes
    \(t_{\delta}(\mu) \in \mc{O}(\kappa d)\), recovering the
    mixing time bound for MALA in
    \cite[Theorem~1]{dwivedi2019log}.
    As shown in Lemma~\ref{lem:temperedAcceptance}, these regimes
    emerge from the localisation having to control two competing
    effects. The \(\mc{O}(\kappa^2)\) scaling stems from the
    stronger localisation required to control the displacement of
    the surrogate mode from the current state in the presence of
    steep target gradients, requiring
    \(\gamma \gtrsim L \kappa\).
    The \(\mc{O}(\kappa d)\) regime arises from the impact
    of the natural fluctuations of \(z \sim \Pi_x\) on the
    acceptance probability and can be controlled by
    \(\gamma \gtrsim L d\).
\end{remark}

\begin{remark}
    \label{rmk:rootChainSampler}
    Using Markov chains to draw proposals
    \(
      z \sim \mc{P}_x = Q_x^{n}\mu_x \approx \Pi_x
    \)
    is a natural choice and allows the proposal quality to be quantified
    in terms of the chain length \(n\). The proof in
    Section~\ref{sec:proof}, however, requires only
    \(
      \tvnorm{\mc{P}_x - \Pi_x} \le \delta/(2 N_\delta)
    \)
    on \(\mc{K}_N\), as seen in \eqref{eq:perturbationSplit}, regardless
    of how the proposal measure is constructed. Any family of samplers
    whose single-draw law satisfies this bound yields the same mixing
    time bound as Theorem~\ref{thm:mixingTime}, with the acceptance
    step~\eqref{eq:dartAcceptance} unchanged. Exact sampling, as in
    Corollary~\ref{cor:exactSampling}, is one such instance. The inner
    sampler can therefore be chosen on purely computational grounds,
    provided the required total variation accuracy can be established,
    in which case the mixing analysis carries over without modification.
\end{remark}

        \subsection{Practical Aspects}       
            \label{sec:dart:practice}
            Theorem~\ref{thm:mixingTime} is in a strongly regularised
regime. This is deliberate, as the resulting strong localisation
confines proposals to a neighbourhood of the current state,
where gradient fidelity is the dominant concern.
In practice, substantially smaller values of \(\gamma\) are
effective, for two reasons.
First, the target's own strong convexity
already provides a degree of localisation that the analysis
neglects. In the proof, the localisation term is taken to
dominate the curvature of \(f\)
(cf.\ Lemmas~\ref{lem:surrogateWarmness}
and~\ref{lem:temperedAcceptance}), but in practice the two
curvatures compound, and a smaller \(\gamma\) suffices.
Second, the fidelity condition~\eqref{eq:fidelityCondition} only
controls first-order information. If the surrogate also captures
curvature information of the target, the localisation would
only need to compensate for the remaining curvature mismatch.
The theorem should therefore be read as a rigorous
illustration of the principle that localised surrogate
transitions can extract gradient-level geometric information
from the surrogate, rather than as a prescription of optimal
algorithmic parameters.

The acceptance probability \eqref{eq:dartAcceptance} involves
the ratio \(N_x / N_z\) of normalisation constants. When
\(g\) is quadratic, say
\(g(y) = \tfrac{1}{2}(y - \hat{x})^\top A \, (y - \hat{x})\)
for a positive-definite matrix~\(A\), the surrogate
density \(\pi_x\) is Gaussian with precision
\(\theta A + \gamma \mbb{I}\). The normalisation constant
\(N_x\) then follows from the determinant of this precision,
and the ratio \(N_x / N_z\) is available in closed form.
This covers, for instance, Laplace approximations, where \(A\) is the Hessian
at the MAP, and the high-frequency components of the DNA
surrogate in Section~\ref{sec:sglmm}, where \(A\) is a
diagonal matrix of prior precision eigenvalues.

When \(g\) is not quadratic, the ratio must be estimated. The
difference of potentials
\begin{equation}    \label{eq:exponentLinearity}
    V_{z}(v) - V_{x}(v)
        = \gamma \iprod{v - \tfrac{x+z}{2},\, x-z}
\end{equation}
is linear in \(v\), since the quadratic terms cancel.
This allows the reciprocal ratio to be expressed as
\begin{equation}    \label{eq:normConstRewrite}
    \frac{N_z}{N_x}
        = \exv_{v \sim \Pi_x}\,
            e^{-\gamma \iprod{
                v - \frac{x+z}{2},\, x-z
            }}.
\end{equation}
The root chain already generates samples approximately
distributed according to \(\Pi_x\), so these can be repurposed
to estimate \eqref{eq:normConstRewrite} at no additional model
cost. Given \(\tilde{n}\) effectively independent states
\(\{v_{i}\}_{i=1}^{\tilde{n}}\) from the root chain, the estimator
\begin{equation}    \label{eq:isEstimator}
    \hat{R}_{\mrm{IS}}^{-1}(x, z)
        \coloneqq \frac{1}{\tilde{n}}
            \sum_{i=1}^{\tilde{n}} e^{
                -\gamma \iprod{v_i - \frac{x+z}{2},\, x-z}
            }
\end{equation}
is unbiased for \(N_z / N_x\), requires no additional target
evaluations, and is exact when \(x = z\). Taking the
reciprocal
\(\hat{R}_{\mrm{IS}}(x,z) \coloneqq
    1 / \hat{R}_{\mrm{IS}}^{-1}(x,z)
\)
introduces an upward bias of order
\(\mc{O}(1/\tilde{n})\) by Jensen's inequality.

Equation~\eqref{eq:isEstimator} is the empirical mean of \(e^{-\ell}\), with
\(\ell(v) \coloneqq \gamma \iprod{v - \tfrac{x+z}{2},\, x-z}\).
A second estimator follows from approximating \(\Pi_x\) by the
Gaussian derived from the chain's empirical first two cumulants.
This too costs no further evaluations and gives
\begin{equation}    \label{eq:cumulantEstimator}
    \log \hat{R}^{-1}_{\mrm{G}}(x, z)
        \coloneqq -\bar{\ell} + \tfrac12 s_\ell^2,
\end{equation}
with the sample mean \(\bar{\ell}\) and variance \(s_{\ell}^2\). It estimates the
log-ratio directly, avoiding the exponentiation and the reciprocal
step that produced the Jensen bias of \eqref{eq:isEstimator}.

The two estimators trade consistency against efficiency.
Estimator \eqref{eq:cumulantEstimator} is exact whenever
\(\Pi_x\) is Gaussian, recovering the closed-form quadratic case,
and returns zero at \(x = z\). For skewed, non-Gaussian \(\Pi_x\) it
retains a bias that does not vanish as \(\tilde{n} \to \infty\).
The importance-sampling estimator \eqref{eq:isEstimator} is
instead consistent, at higher variance and with less stable
exponential weighting. The two estimators are compared
empirically in Section~\ref{sec:numerics}.

When the surrogate \(g\) is itself a Bayesian posterior with
a known Gaussian prior of precision \(C^{-1}\), the full Gaussian
component of \(V_x\) has precision \(\theta \, C^{-1} + \gamma \mbb{I}\),
providing a natural pCN reference measure \cite{cotter2013mcmc}. Tempering
only the likelihood component of \(g\) in such instances, while leaving the
prior at full strength, preserves the prior geometry and ensures that
the localisation remains solely responsible for confining
proposals. This is employed in Sections~\ref{sec:numerics}
and~\ref{sec:sglmm}.

In practice, we initialise the root chain at the current outer
state, setting \(\mu_x = \delta_x\). Although \(\delta_x\) is
not a warm start for \(\Pi_x\), the transition operator
\(Q_x\) yields an absolutely continuous measure
\(Q_x \delta_x\) after one step, which, provided the
localisation is sufficient, concentrates adequately to serve as
an approximately warm start. We also implement a \emph{crank-up}
procedure: an auxiliary chain targeting \(\widetilde{\pi} \propto e^{-g}\)
is run for a fixed number of steps, and its terminal state is used as
the initial state of the outer chain, giving \(\mu \approx \widetilde{\Pi}\).

    \section{Empirical Behaviour}
        \label{sec:numerics}
        The experiments in Sections~\ref{sec:numerics} and~\ref{sec:sglmm}
rely on \texttt{styne}, a standalone Python library available at
\url{https://github.com/rkutri/styne}. The scripts reproducing all
experiments and figures in this paper reside under
\texttt{reproducibility/dart} in release~\texttt{v0.1.0}, browsable at
\url{https://github.com/rkutri/styne/tree/v0.1.0}.

        \subsection{Stability of the Ratio Estimators}
            \label{sec:numerics:stability}
            We isolate the accuracy of the estimation step on a two-component
Gaussian mixture with unequal weights. The projected law of the root
chain samples is skewed, so the experiment exercises
\(\hat R_{\mathrm{G}}\) outside the Gaussian regime in which it is
exact.
We set \(\theta = 1\), since tempering would destroy the mixture form
and with it the analytic ground truth. Localisation is therefore
controlled by \(\gamma\) alone.

We fix \(x\) at the mode of the dominant component, draw independent
realisations of \(z \sim Q_x^n \mu_x\), where \(Q_x\) is MALA on
\(\Pi_x\) initialised at \(\mu_x = \delta_x\), and compute
\(\hat R_{\mathrm{IS}}(x, z)\) and \(\hat R_{\mathrm{G}}(x, z)\) from
the same trajectory.

Figure~\ref{fig:stability} reports the absolute error in
\(\log(N_z/N_x)\) as a function of root
chain length~\(n\), across dimensions and localisation strengths. As
\(\gamma\) grows, \(\Pi_x\) and \(\Pi_z\) tighten and their overlap
shrinks, so the ratio is harder to estimate, a difficulty that
compounds with dimension. The two estimators agree while sampling is cheap.
Once the log-weight variance grows, \(\hat R_{\mathrm{IS}}\) degrades
under large weights while \(\hat R_{\mathrm{G}}\) does not. This
is the trade-off of Section~\ref{sec:dart:practice} made visible, and
it is why we use \(\hat R_{\mathrm{G}}\) as the default estimator.
While \(\hat R_{\mathrm{G}}\) is more reliable in this experiment,
it is expected to be more vulnerable to surrogate geometries that differ more
strongly from its Gaussian approximation.

How much residual estimation noise
affects DART's overall mixing depends on the relative sizes of all
terms in \eqref{eq:dartAcceptance}, not on the ratio error alone. The
subsequent experiments confirm that, once \(\gamma\) is chosen
appropriately, DART improves over its non-localised counterparts.

\begin{figure}[!htbp]
    \centering
    \includegraphics[width=\textwidth]{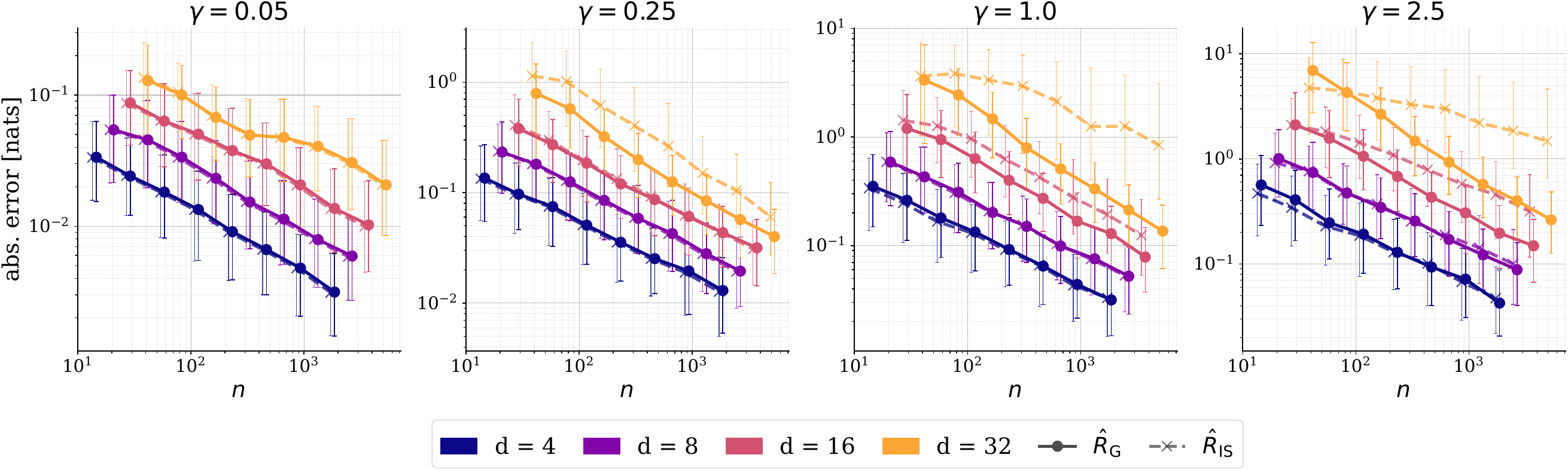}
    \caption{
        Median absolute error in \(\log(N_z/N_x)\), with
        interquartile range over \(10\,000\) independent runs, for
        \(\hat{R}_{\mrm{IS}}\) (dashed) and \(\hat{R}_{\mrm{G}}\)
        (solid) as a function of root chain length~\(n\). The
        mixture has weights \((0.7, 0.3)\), means
        \(\pm a\mathbb{1}\) with \(a = 0.5\), common covariance
        \(\sigma^2 \mbb{I}\) with \(\sigma = 1\), and \(\theta = 1\).
        The root chain is MALA on \(\Pi_x\) initialised at
        \(\mu_x = \delta_x\), burn-in fraction \(0.3\), no thinning.
    }
    \label{fig:stability}
\end{figure}

        \subsection{The Quadratic Surrogate Regime}
            \label{sec:numeric:quadratic}
            We now test the full DART chain in the regime of
Corollary~\ref{cor:exactSampling}, where the surrogate \(g\) is
quadratic and proposals can be drawn from \(\Pi_x\) directly.
Following \cite{dalalyan2017theoretical, dwivedi2019log}, we consider
a Bayesian logistic regression example. Given \(n_{\mrm{obs}}\) observations
\((c_i, y_i) \in \R^d \times \{0,1\}\), the probability that
\(y_i = 1\) is modelled as \(\sigma(\iprod{x, c_i})\), with
\(\sigma(t) = (1 + e^{-t})^{-1}\). Covariates \(c_i\) have i.i.d.\
Rademacher entries and unit Euclidean norm and responses are drawn from a
fixed data-generating parameter \(x^{\star}\). The prior is
\(\mc{N}(0, (\alpha\Sigma_C)^{-1})\) with
\(\Sigma_C = \tfrac{1}{n_{\mrm{obs}}} C^\top C\) and parameter
\(\alpha\), whose value can be chosen freely.

After preconditioning by \(\Sigma_C^{-1/2}\), the prior becomes
\(\mc{N}(0, \alpha^{-1}\mbb{I}_d)\), the preconditioned design matrix
satisfies \(\tilde{C}^\top \tilde{C} = n_{\mrm{obs}}\mbb{I}_d\), and the
negative log-posterior is
\begin{equation}\label{eq:logisticTarget}
  f(x) = -Y^\top \tilde{C}\, x + \sum_{i=1}^{n_{\mrm{obs}}}
  \log\big(1 + e^{\iprod{x, \tilde{c}_i}}\big)
  + \frac{\alpha}{2} \enorm{x}^2,
\end{equation}
with gradient Lipschitz constant \(L = \tfrac{1}{4}n_{\mrm{obs}} +
\alpha\), strong convexity constant \(\lambda = \alpha\), and condition
number \(\kappa = L/\lambda\), independent of the spectrum of the
original design matrix. With \(n_{\mrm{obs}}\) and \(\alpha\) fixed,
both \(L\) and \(\kappa\) are the same at every dimension.

As surrogate for \(f\) we use the Laplace approximation at the MAP
\(\hat{x} = \argmin{\R^d} f\),
\begin{equation}\label{eq:laplaceSurrogate}
  g(y) = \tfrac{1}{2}\iprod{y - \hat{x},\; H\,(y - \hat{x})},
\end{equation}
where \(H\) is the Hessian of \(f\) at \(\hat{x}\). Since \(g\) is
quadratic, the localised surrogate density is Gaussian with precision
\(P = \theta H + \gamma\mbb{I}_d\) and mean
\(P^{-1}(\theta H\hat{x} + \gamma x)\). Both \(P\) and the normalisation
ratio \(N_x / N_z\) are available in closed form, so no root chain or
ratio estimation is needed.

We compare DART against MALA, MLDA and MRW in the preconditioned space.
MALA and MRW step sizes are tuned to standard target acceptance
rates and MLDA shares DART's surrogate, but canonically, without regularisation
or tempering. The implicit step size of DART is governed by \(\gamma\) and is
swept over \(\gamma / L \in [10^{-2}, 10]\). The value at the peak of the
\enquote{Efficiency} panel of Figure~\ref{fig:logistic} is used for DART. To
assess convergence, we track the running mean diagnostic
\begin{equation}\label{eq:meanDiagnostic}
  e_k = \frac{1}{d}\norm{\bar{x}_k - x_{\mrm{ref}}}_1,
  \qquad
  \bar{x}_k = \frac{1}{k}\sum_{i=1}^{k} x_i,
\end{equation}
where \(x_{\mrm{ref}} = \Sigma_C^{1/2} x^{\star}\) is the data-generating
parameter in the preconditioned coordinates. Since every chain targets
the same posterior, \(e_k\) decays to a common floor, the distance
between the posterior mean and \(x_{\mrm{ref}}\). All chains share common
initialisations drawn from the prior. To compare stationary-phase
efficiency, we report effective sample size (ESS) per iteration along the
slowest direction of the target, the eigenvector of \(H\) with the smallest
eigenvalue. We project \(M = 8\) replicate chains, each started from an
independent prior draw, onto this direction and combine them with a
multi-chain ESS estimator that pools the within-chain autocorrelation with
the between-chain variance (c.f.~\cite{gelman1992inference}).

Figure~\ref{fig:logistic} summarises the results. MLDA employs the same surrogate
as DART but lacks localisation. As its sub-chain equilibrates, the proposals
approach independent draws from the Laplace approximation. The left panel
demonstrates that this measure lacks the fidelity required to serve the
full model. The success of this same surrogate under DART isolates the
effect of the localisation $\gamma$ in keeping each proposal local. Finally,
MRW converges slowly but reliably.

DART converges faster than MALA, without evaluating the target
gradient at any iteration. In line with the discussion in Section~\ref{sec:dart:practice},
when in a beneficial localisation regime (see middle panel in Figure~\ref{fig:logistic}),
curvature information captured by the Laplace approximation further improves DART
over MALA, which only has access to gradient information. The middle panel of Figure~\ref{fig:logistic} illustrates the
theoretical finding that the optimal localisation strength \(\gamma\) is
determined by the gradient-Lipschitz constant \(L\), the largest curvature of
the target. For this experiment, optimal ESS is achieved consistently
for \(\gamma \approx 0.2\,L\). The rightmost panel indicates the
mechanism behind this: increasing \(\gamma\) shrinks the proposal scale and raises
acceptance monotonically. Below the peak, proposals overshoot and are
rejected. Above it, proposals are accepted but take steps too small to
decorrelate efficiently.
\begin{figure}[!htbp]
  \centering
  \includegraphics[width=\textwidth]{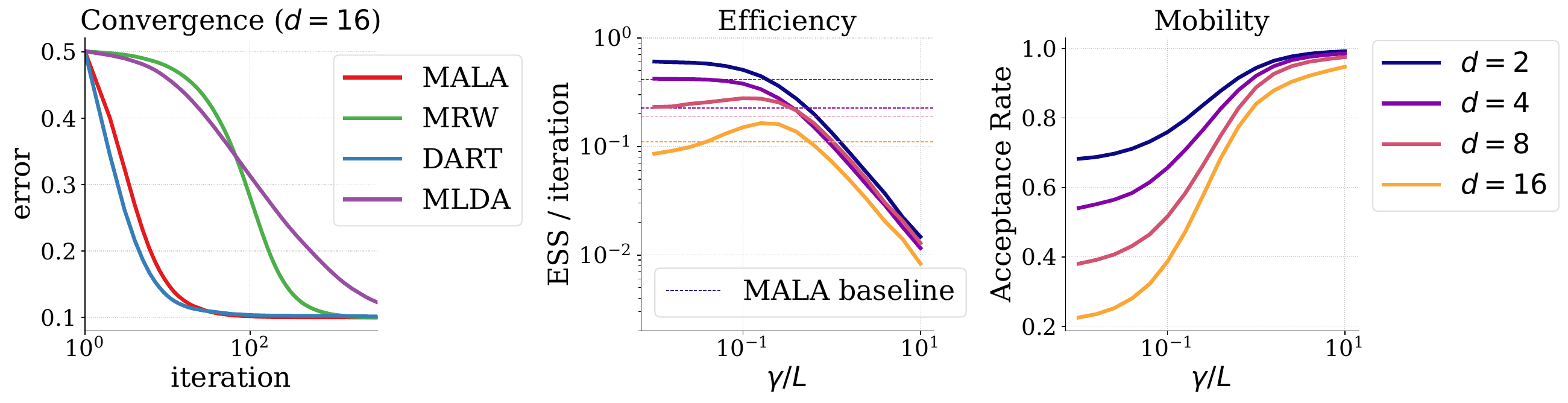}
  \caption{Bayesian logistic regression with a Laplace surrogate
  (\(n_{\mrm{obs}} = 240\), \(\alpha = 3\), giving \(\kappa = 21\);
  \(\theta = 1/2\)). (a)~Running mean error \(e_k\) at \(d = 16\) for
  MALA, MRW, MLDA, and DART at \(\gamma = 0.2\,L\), averaged over
  \(n_{\mrm{est}} = 500\) independent runs from prior initialisations.
  (b)~ESS per iteration of DART along the slowest target direction, as a
  function of \(\gamma / L\) for dimensions \(d \in \{2, 4, 8, 16\}\),
  estimated from \(M = 8\) replicate chains. Horizontal dashed lines
  show the same quantity for tuned MALA at matching dimensions.
  (c)~Acceptance rate of DART as a function of \(\gamma / L\). Target
  acceptance rates for tuning were \(0.55\) for MALA and \(0.25\) for
  MRW.}
  \label{fig:logistic}
\end{figure}

        \subsection{Multimodal Targets}
            \label{sec:numerics:multimodal}
            Departing from the log-concave setting of Theorem~\ref{thm:mixingTime},
we investigate DART's behaviour on multimodal targets. Multimodality is a
fundamental obstacle for MCMC methods that rely on local proposals, as the
chain can become trapped in a single mode for extended periods.
In DART, tempering the surrogate density reduces the effective depth of these barriers,
making transitions that the outer chain would reject accessible to the
surrogate chain. This is the same mechanism exploited by simulated
annealing, but in DART it is built into the proposal step itself,
requiring no separate tempered chains or swap moves.
When a single level of tempering is insufficient to overcome the barrier,
DART admits a recursive extension in the spirit of MLDA. Each level
employs a smaller tempering parameter, producing a progressively flatter
surrogate landscape. The resulting temperature ladder is traversed
sequentially within a single outer proposal step.

We illustrate this mechanism on two bimodal Gaussian mixture targets
with mode separation \(\Delta\) and per-mode variance \(\sigma^2\). To
isolate the barrier-crossing effect from surrogate approximation error,
we set \(g = f\) throughout. For two-level DART, each outer step then
requires \(n + 1\) target evaluations: one for the acceptance step
and \(n\) for the surrogate chain. For three-level DART it is correspondingly
even higher.
The purpose of this experiment is to demonstrate the mixing mechanism,
not to benchmark computational cost.

We compare MRW, MALA, two-level DART, and three-level DART. The DART
variants use pCN, preconditioned against the quadratic penalty term, as
the root chain. The regularisation \(\gamma\) is chosen per target so
that the localised surrogate density spans roughly the combined support
of both modes.
Figure~\ref{fig:bimodal} shows trace plots alongside the target
density (grey) and, for the DART variants, the localised surrogate
densities at each tempering level (dashed, coloured), evaluated at
\(x = 0\).
Each trace is annotated with its integrated autocorrelation time
estimate and its number of mode transitions.

On the moderate target (\(\Delta = 4\), \(\sigma^2 = 0.4\)) MRW
takes steps larger than the mode separation and tunnels across the 
barrier, transitioning freely. MALA
crosses only a handful of times, since its gradient drift pulls each
proposal toward the nearest mode. The drift that speeds within-mode
mixing impedes transitions between modes, and its autocorrelation time
is large as a result. Both DART variants cross freely. On the harder
target (\(\Delta = 8\), \(\sigma^2 = 0.2\)), the modes are too far
apart to tunnel at any feasible step size, so MRW fails alongside MALA,
each confined to a single mode. Two-level DART still transitions but
with visible periods of entrapment, while three-level DART mixes
freely, showing that the flatter additional level is what carries the
chain across the deeper barrier.
\begin{figure}[!htbp]
    \centering
    \includegraphics[width=\textwidth]{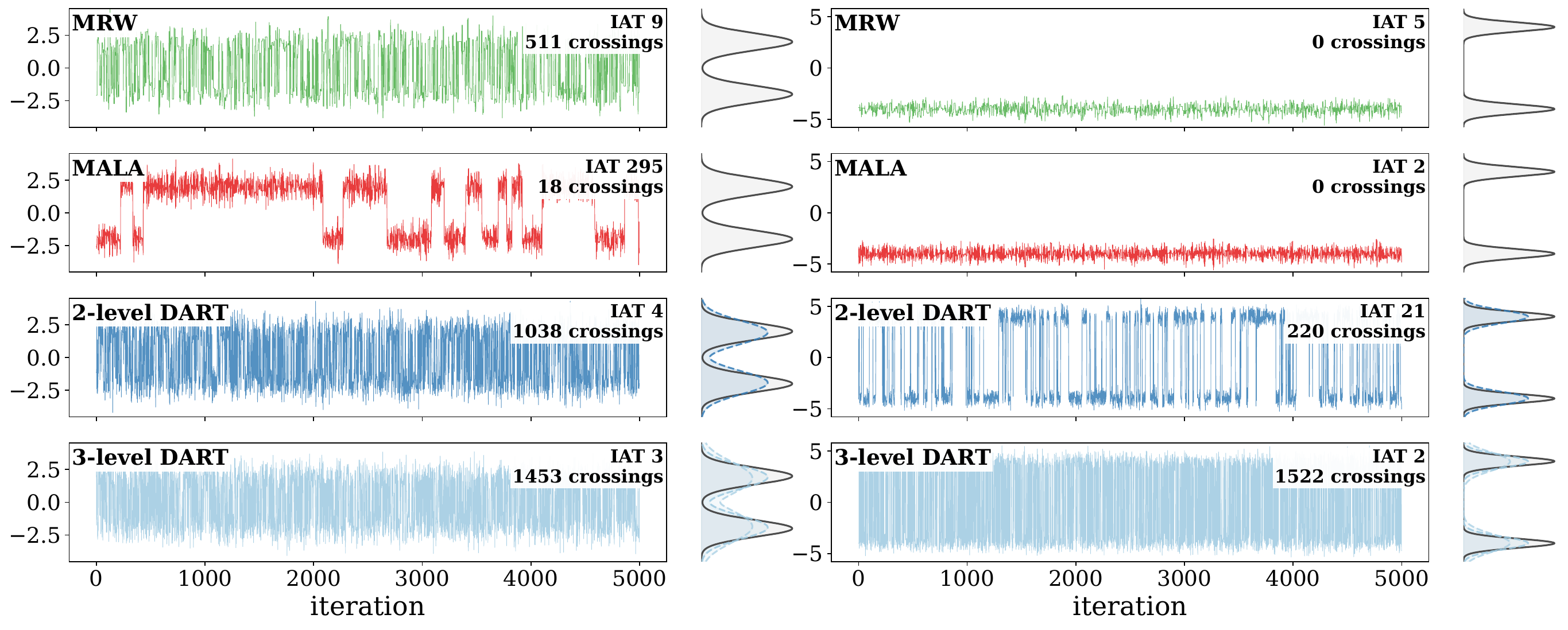}
    \caption{%
        Trace plots (first 5\,000 post-burn-in iterations) for MRW, MALA,
        two-level DART, and three-level DART on bimodal Gaussian mixture targets with
        \((\Delta,\,\sigma^2) = (4,\,0.4)\) (left) and
        \((\Delta,\,\sigma^2) = (8,\,0.2)\) (right).
        Side panels show the target density (grey) and, for the DART
        variants, the localised surrogate densities at each tempering
        level (dashed, coloured), evaluated at \(x = 0\).
        Each trace is annotated with its IAT (estimated from
        \(50\,000\) post-burn-in iterations) and mode-transition count.
        Acceptance tuned to MRW~\(0.3\), MALA~\(0.5\), and the pCN
        root~\(0.3\).
        Two-level DART: \(\theta = 0.5\), \(n = 20\).
        Three-level DART: \(\theta_1 = 0.25\), \(\theta_2 = 0.5\),
        \(n_1 = 20\), \(n_2 = 10\).
        Regularisation: \(\gamma = 0.05\) (moderate),
        \(\gamma = 0.01\) (harder).
    }
    \label{fig:bimodal}
\end{figure}

    \section{Spatial Generalised Linear Mixed Models}
        \label{sec:sglmm}
        We now present the novel Gaussian process parametrisation based on 
Dirichlet-Neumann Averaging (DNA) \cite{kutri2026dirichlet} and
show its synergy with DART on a hierarchical SGLMM problem.
For background on SGLMMs and
extensions to other observation models, see
\cite{diggle1998model, breslow1993approximate}.
        \subsection{Latent Field Inference}
        \label{sec:sglmm:latent}
        We consider a SGLMM with Poisson
response. Let the spatial domain be
\(\Omega \coloneqq (0,1)^D \subset \R^D\), with \(N\)
observation locations
\(\{s_i\}_{i=1}^{N} \subset \Omega\).
At each location we observe the count
\(y_i \in \N_0\) of a spatially varying phenomenon.
The underlying intensity is modelled by a latent,
centred GP \(u\) on \(\Omega\).
Conditional on this \emph{latent field}, observations are
assumed independent and Poisson distributed:
\begin{equation}\label{eq:sglmm}
    y_i \given u(s_i)
        \sim \mathrm{Poisson}\bigl(e^{\eta_i}\bigr),
        \qquad
        \eta_i \coloneqq \beta_0 + u(s_i),
        \qquad i = 1, \ldots, N,
\end{equation}
where \(\beta_0 \in \R\) is a known mean log-intensity.

Covariates are omitted from~\eqref{eq:sglmm} to
streamline the exposition.
We equip the latent field \(u\) with a GP prior
whose covariance function is Mat\'ern with marginal
variance \(\sigma^2\), correlation length \(\rho\),
and fixed smoothness \(\nu = 5/2\), and collect
the free covariance parameters in
\(\psi = (\rho, \sigma^2)\).
Together with the Poisson likelihood arising from~\eqref{eq:sglmm},
this prior defines a posterior density over \(u\).
As the Poisson likelihood is not conjugate to the
GP prior, closed-form updates are unavailable and
MCMC transitions are required.
In Section~\ref{sec:sglmm:hierarchical} we extend this
model by placing a hyperprior \(p(\psi)\) on
\(\psi\) and targeting the full joint posterior
over \(u\) and \(\psi\).

MCMC transitions require repeated evaluations of
the latent field \(u\) at the \(N\) observation
sites.
A naive Cholesky parametrisation of the latent
field values assembles and
factors the \(N \times N\) covariance matrix at
cost \(O(N^3)\), making each such evaluation
prohibitively expensive even at moderate \(N\).

\subsubsection{Latent Field Parametrisation}
\label{sec:dna-param}
The DNA framework of \cite{kutri2026dirichlet} represents the
latent field~\(u\) as the scaled
average of \(2^D\) independent GPs,
each satisfying a different combination of homogeneous
Dirichlet and Neumann conditions on
\(\partial\Omega\).
The boundary artifacts introduced by any single choice
of boundary conditions cancel exactly under this
averaging, yielding a genuinely isotropic field on
\(\Omega\). We refer
to~\cite{kutri2026dirichlet} for details in simulation and
error analysis, and present mainly the novel elements
pertaining to inference.

Boundary conditions are encoded in a vector
\(b \in \{0,1\}^D\), where \(b_j = 0\) prescribes
Neumann and \(b_j = 1\) Dirichlet conditions on the
\(j\)-th pair of opposing faces.
For a spectral resolution \(q \in \N\), define the
index sets \(I_b = I_{b_1} \times \cdots \times I_{b_D}\)
with \(I_0 = \{0,\ldots,q\}\) and
\(I_1 = \{1,\ldots,q\}\).
Let \(\hat\varphi_\psi\) denote the Fourier
transform of the Mat\'ern covariance function with
parameters \(\psi = (\rho, \sigma^2)\)
(see~\cite[Eq.~(2.22)]{graham2018analysis}), and
write
\(\omega_\nu = \omega_\nu(\psi)
    \coloneqq \sqrt{\hat\varphi_\psi(\nu/2)}\)
for the spectral weights at frequency~\(\nu\).
The DNA parametrisation of \(u\) is then given by
\begin{equation}\label{eq:dna-field}
    u(s,\,x)
        = 2^{-D/2}
          \sum_{b \in \{0,1\}^D}
          \sum_{\nu \in I_b}
          x_\nu^b
          \omega_\nu
          e_\nu^b(s),
    \qquad s \in \Omega,
\end{equation}
where the basis functions \(e_\nu^b\) are tensor
products of cosine (\(b_j = 0\)) and sine
(\(b_j = 1\)) modes, and the parameters
\(x = (x^b)_{b \in \{0,1\}^D} \in \R^d\),
with total dimension \(d = (2q+1)^D\), have prior
\(x \sim \mc{N}(0, I_d)\) (for details, see
\cite[Section~3]{kutri2026dirichlet}). 
Since each of the \(2^D\) component fields carries
its own coefficient vector but is evaluated on the
same grid, \(d\) exceeds the number of grid points
\(Q = (q+2)^D\) by a constant factor.

To evaluate the latent field at the \(N\) observation
sites, we introduce an equispaced grid \(\mc{G}\) on
\(\Omega\) with \(q+2\) vertices per dimension.
On this grid, each inner sum
in~\eqref{eq:dna-field} is computed via the
appropriate combination of discrete cosine and sine
transforms, encoded in a matrix \(W_b\).
A sparse bilinear interpolation matrix
\(\mrm{I}_\mrm{F} \in \R^{N \times Q}\)
maps the grid values to the observation sites
(c.f.~\cite{graham2015quasi}), giving
\begin{equation}\label{eq:field-eval}
    \bm{u}(x)
        = 2^{-D/2} \,\, \mrm{I}_\mrm{F} 
          \Biggl(
            \sum_{b \in \{0,1\}^D}
            W_b \,
            \Lambda_b \,
            x^b
          \Biggr)
        \;\in\; \R^N,
\end{equation}
where
\(
    \Lambda_b
        \coloneqq \diag\bigl(
            \omega_\nu
        \bigr)_{\nu \in I_b}
\)
and
\(
    \bm{u}_i(x) \coloneqq u(s_i, x).
\)
The total cost per evaluation is therefore
\(\mc{O}(Q \log Q + N)\).

\subsubsection{Surrogate Construction}
\label{sec:dart-sglmm}

For fixed hyperparameters \(\psi\), the posterior
over the latent field in its DNA parametrisation is
\(\pi(x) \propto e^{-f(x)}\) with potential
\begin{equation}\label{eq:target-potential}
    f(x)
        = -\mc{L}(x) + \frac{1}{2} \enorm{x}^2,
    \qquad
    \mc{L}(x)
        \coloneqq \sum_{i=1}^{N}
          \bigl(
            y_i \, \eta_i(x)
            - e^{\eta_i(x)}
          \bigr),
\end{equation}
where \(\mc{L}\) is the Poisson log-likelihood and
\(\eta_i(x) \coloneqq \beta_0 + \bm{u}_i(x)\) is the
linear predictor at site \(s_i\).
The gradient of \(f\) with respect to each block
\(x^b\) is available at cost
\(\mc{O}(Q \log Q + N)\) via the
adjoint of the interpolation and transform operators.

The structure of~\eqref{eq:dna-field}
provides a natural partition of the parameters.
Fix a coarse resolution \(q_\mrm{C} < q\) and define
coarse index sets \(I_{b,\mrm{C}}\) by replacing \(q\)
with \(q_\mrm{C}\). Each block decomposes as
\(x^b = (x^b_\mrm{C},\, x^b_\mrm{F})\),
where \(x^b_\mrm{C}\) collects modes with indices
in \(I_{b,\mrm{C}}\) and \(x^b_\mrm{F}\) collects
the remainder.
A coarse approximation to the log-likelihood is
obtained by replacing the field evaluation
in~\eqref{eq:target-potential} with its coarse-grid
analogue, using \(W_{b,\mrm{C}}\) and
\(\mrm{I}_\mrm{C}\) in place of \(W_b\) and
\(\mrm{I}_\mrm{F}\), and acting only on
\(x_\mrm{C}\). We denote this coarse log-likelihood
by \(\mc{L}_\mrm{C}(z_\mrm{C})\).
Provided \(q_\mrm{C}\) is chosen large enough that
the high-frequency modes carry negligible likelihood
information, a natural choice is to let these modes
contribute only via the prior. When at state \(x = (x_\mrm{C}, x_\mrm{F})\),
the localised surrogate potential used for DART is
\begin{equation}\label{eq:vXDART}
    V_x(z)
        = -\theta \, \mc{L}_\mrm{C}(z_\mrm{C})
          \;+\; \frac{\gamma}{2}
                \enorm{z_\mrm{C} - x_\mrm{C}}^2
          \;+\; \frac{1}{2} \enorm{z}^2,
\end{equation}
where
\(\theta \in (0,1)\) tempers the coarse likelihood and
the quadratic penalty localises only the coarse
component. Since the parameters are independent
under the prior, the prior density factorises
across the partition, giving the
\(
    \enorm{z_\mrm{C}}^2 + \enorm{z_\mrm{F}}^2
        = \enorm{z}^2 
\)
decomposition in~\eqref{eq:vXDART}.
A proposal \(z = (z_\mrm{C}, z_\mrm{F})\) is drawn
by running a root chain targeting the coarse
component of \(\Pi_x\), giving \(z_\mrm{C}\), and
drawing \(z_\mrm{F}\) from a pCN step targeting
\(\mc{N}(0, I_{d_\mrm{F}})\).

The acceptance probability~\eqref{eq:dartAcceptance} for this
construction simplifies considerably. We track three
cancellations in turn.
\begin{itemize}
    \item \textit{Quadratic penalty.} As in
    Section~\ref{sec:dart}, the symmetry
    \(V_x(z) - V_z(x) = -\theta (\mc{L}_\mrm{C}(z_\mrm{C}) -
    \mc{L}_\mrm{C}(x_\mrm{C})) + \tfrac{1}{2}(\enorm{z}^2 - \enorm{x}^2)\)
    eliminates the quadratic localisation term from the proposal density
    ratio; the prior difference that remains is cancelled below.

    \item \textit{High-frequency block.} Since \(x_\mrm{F}\) is not
    localised in \(V_x\), the dependence of the normalisation
    constant on the state enters only through the low-frequency
    component. We leverage
    \(N_x = N_{x,\mrm{C}} \cdot C_\mrm{F}\), where \(C_\mrm{F}\) is
    the Gaussian integral over \(z_\mrm{F}\) and is independent
    of \(x\). The ratio reduces to
    \(N_x / N_z = N_{x,\mrm{C}} / N_{z,\mrm{C}}\), which is a problem
    in significantly lower dimension than then initial one.

    \item \textit{Low-frequency prior.} Both \(f\) and \(V_x\) carry the
    prior term \(\tfrac{1}{2} \enorm{\cdot}^2\). The contribution
    to \(\pi(z) / \pi(x)\) is \(\exp(-\tfrac{1}{2}(\enorm{z}^2 -
    \enorm{x}^2))\). The contribution to \(p_z(x) / p_x(z)\) is its
    inverse, and the two cancel. This cancellation is structurally 
    the same as in pCN \cite{cotter2013mcmc}, with the localisation
    parameter \(\gamma\) playing no role in the preconditioned step.
\end{itemize}

Incorporating these three cancellations, the acceptance probability
reduces to
\begin{equation}\label{eq:alpha-dart-sglmm}
    \alpha(x, z) = \min \left\{
            1,~
            \exp \Big(
                \mc{L}(z) - \mc{L}(x)
                    + \theta \big(
                        \mc{L}_\mrm{C}(x_\mrm{C})
                            - \mc{L}_\mrm{C}(z_\mrm{C})
                    \big)
            \Big) \frac{N_{x,\mrm{C}}}{N_{z,\mrm{C}}}
    \right\}.
\end{equation}

\subsubsection{Experiment}
\label{sec:1d-experiments}

The Hyperparameters are fixed at \(\psi = (\rho, \sigma^2) = (0.1, 1.25)\)
and we consider \(D = 1\) spatial dimension. We begin in a regime where
the coarse surrogate captures most likelihood information at the
chosen correlation length. We compare DART against three alternatives:
MALA on the posterior under a Cholesky parametrisation, MALA on the
DNA parametrisation~\eqref{eq:target-potential} of the full posterior,
and MLDA on the DNA parametrisation using the same surrogate as DART
but without regularisation or tempering.
We deliberately do not include the INLA \cite{rue2009approximate}
and stan \cite{carpenter2017stan} HMC software routines in the
comparison, as these occupy different points on the trade-off between
cost, accuracy, and scalability.

Figure~\ref{fig:1d-figA} presents three diagnostic views. The worst-
and average-case IAT bars reveal a monotone hierarchy, from Cholesky,
to DNA with MALA, to DNA with MLDA, to DNA with DART. The DNA
parametrisation alone reduces the worst-case IAT over Cholesky; the
surrogate hierarchy in MLDA buys a further factor over MALA on DNA, and
regularisation and tempering bring DART to the lowest IAT. The
best-case bars do not follow this order, nor is this expected: the
best-mixing location is governed by the high-frequency blocks
\(x_\mrm{F}\), which carry negligible likelihood information and so mix
at essentially the prior rate for every DNA variant. The best-case
summary therefore reflects prior mixing common to all three DNA methods
and cannot separate them, but gives a clear advantage over the 
Cholesky parametrisation. The posterior mean and 95\% credible band
from DART (right panel) track the ground truth closely.
\begin{figure}[!htbp]
    \centering
    \includegraphics[width=\textwidth]{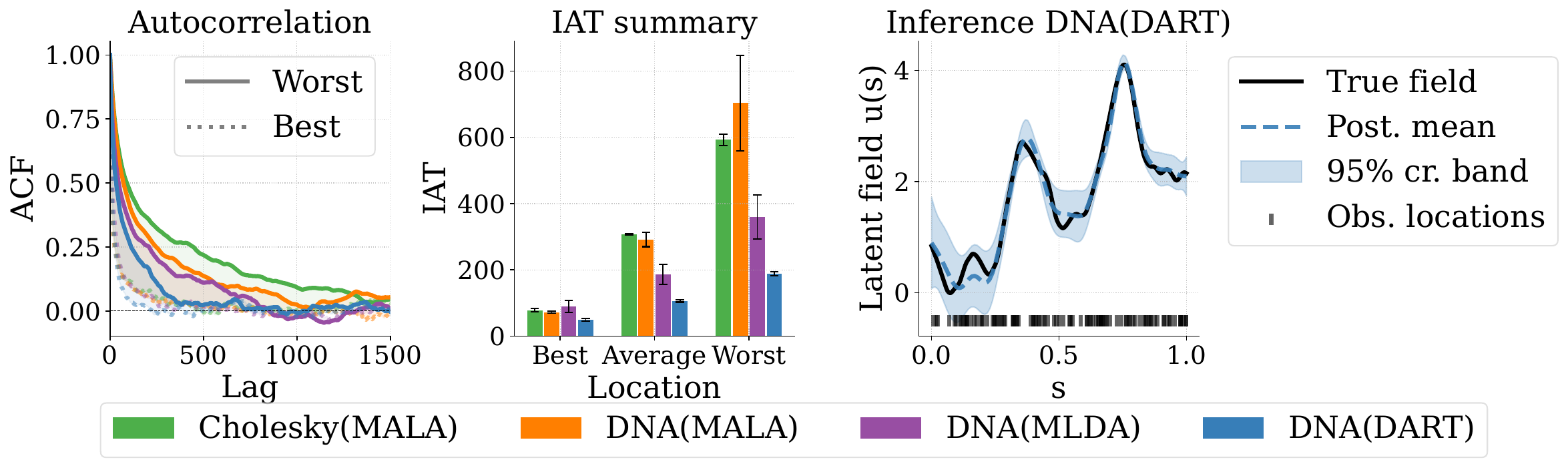}
    \caption{%
        Good-surrogate regime. Ground truth
        is generated once at high resolution
        (\(q_\mrm{gt} = 500\)) and held fixed.
        \(N = 100\) observation sites are
        drawn uniformly on \((0,1)\) and Poisson
        counts are simulated
        from~\eqref{eq:sglmm} with intercept
        \(\beta_0 = 1.5\).
        DNA methods use fine resolution
        \(q = 100\) and coarse resolution
        \(q_\mrm{C} = 20\).
        DART and MLDA run \(n = 25\) root
        chain steps per proposal, with DART using
        \(\gamma = 0.25\) and
        \(\theta = 0.75\).
        Each method is run for \(150\,000\)
        iterations over \(3\) independent replications.
        The respective first \(50\,000\) iterations are 
        discarded as burn-in.
        Left: ACF at worst-case (solid) and
        best-case (dashed) spatial locations.
        Centre: IAT summary (best, average, worst);
        error bars show standard errors.
        Right: posterior mean and 95\% credible band
        from DNA with DART, with ground truth
        (black) and observation locations (ticks).
    }
    \label{fig:1d-figA}
\end{figure}

        \subsection{Hierarchical Inference}
            \label{sec:sglmm:hierarchical}
            We extend the SGLMM from Section~\ref{sec:sglmm:latent} to joint inference
over the latent field and the covariance hyperparameters
\(\psi = (\rho, \sigma)\), now in \(D = 2\) spatial dimensions.
We place a hyperprior on \(\psi\) and use a Metropolis-within-Gibbs sampler
alternating between the latent and hyperparameter blocks.

% --- Hierarchical extension ------------------------------------------------

Placing a hyperprior \(p(\psi)\) on the covariance parameters, the full
joint posterior has density \(\pi(x, \psi) \propto e^{-f(x,\psi)}\)
with potential
\begin{equation}\label{eq:joint-potential}
  f(x, \psi)
    = -\mc{L}(x, \psi) + \tfrac{1}{2}\enorm{x}^2 - \log p(\psi),
\end{equation}
extending~\eqref{eq:target-potential} to include the hyperprior.
Here \(\mc{L}(x, \psi)\) is the Poisson log-likelihood
from~\eqref{eq:target-potential}, now written with its dependence on
\(\psi\) through the spectral weights \(\omega_\nu(\psi)\)
in~\eqref{eq:dna-field} made explicit.
The full conditionals for the two Gibbs blocks have potentials
\begin{equation}
  f_0(x)
    \coloneqq -\mc{L}(x, \psi)
              + \tfrac{1}{2}\enorm{x}^2,
      \quad \text{and} \quad
  f_1(\psi)
    \coloneqq -\mc{L}(x, \psi) - \log p(\psi),
\end{equation}
where \(f_0\) coincides with~\eqref{eq:target-potential} for
fixed~\(\psi\).
The coupling between the two Gibbs blocks is implicit, mediated through the
spectral weights \(\omega_\nu(\psi)\).

Under DNA, the per-sweep covariance update reduces to an \(\mc{O}(d)\)
spectral-weight recomputation. A Cholesky parametrisation would instead
require a fresh \(\mc{O}(N^3)\) factorisation each sweep, dominating
the per-sweep cost at scale.
The latent block targets the same posterior as
Section~\ref{sec:dart-sglmm}, and the DART transition applies as in
that section, with the surrogate re-evaluated at the current \(\psi\).
Since \(f_1\) is only two-dimensional, we update \(\psi\) on the
log-scale via a random walk MH step with
Robbins--Monro adaptation~\cite{robbins1951stochastic}.
For the hyperprior we use a joint penalised complexity
prior~\cite{simpson2017penalising} on \((\rho, \sigma)\), penalising
departure from a base model of zero variance and infinite correlation
length at rates controlled by
\(\Pr(\rho < \rho_0) = \alpha_\rho\) and
\(\Pr(\sigma > \sigma_0) = \alpha_\sigma\).

We consider the SGLMM from~\eqref{eq:sglmm} on the unit square
\(\Omega = (0,1)^2\) with Mat\'ern covariance and a constant
trend \(\beta_0 = 0.5\). We generate a smooth (\(\nu = 5/2\), \(\rho = 0.1\))
and a rough (\(\nu=3/2\), \(\rho=0.05\)) field as ground truth at
high resolution (\(q_{\mrm{gt}} = 500\)). The aim of the experiment in 
this section is to demonstrate robust inference of these fields and the 
corresponding hyperparameters from the corresponding simulated Poisson count data.
Observation locations are drawn from a Sobol sequence on \(\Omega\),
providing quasi-uniform coverage.

Figure~\ref{fig:hierarchical} shows, for each smoothness, the
ground-truth field, the DART posterior mean and the posterior standard deviation.
The posterior mean recovers the truth in both cases.
The posterior mean is slightly
more regular than the truth. We attribute this to
the penalised-complexity prior reverting towards the less
complex, smoother model where the likelihood is uninformative. To guard against
a single fortunate chain we run each demonstration at three independent seeds
and pool the draws, and the recovered field and its uncertainty are stable
across the three.

\begin{figure}[htbp]
  \centering
  \includegraphics[width=\textwidth]{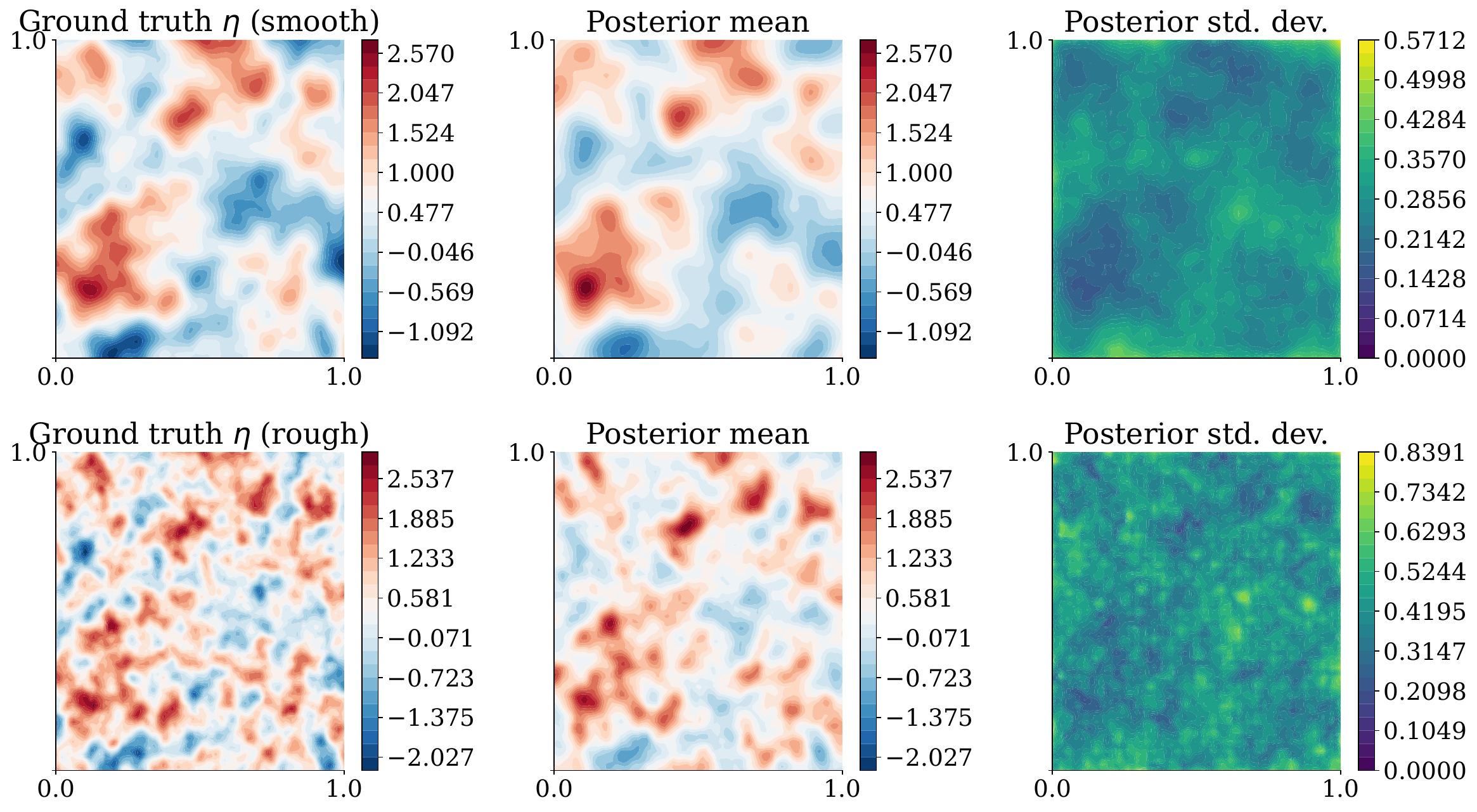}
  \caption{Hierarchical SGLMM, DART posterior fidelity at \(N=1024\). Rows:
    smooth field (\(\nu=5/2\), \(\rho=0.1\), \(\sigma^2=3/4\)) top,
    and rough field (\(\nu=3/2\), \(\rho=0.05\), \(\sigma^2=3/4\)) bottom.
    Columns: ground-truth field, DART posterior mean and posterior
    standard deviation.
    MCMC is run for \(230\,000\) iterations, of which \(30\,000\) are discarded
    as burn-in.
    Fine resolutions are \(q = 64\) for both fields 
    (\(16\,641\) parameters total)
    and the coarse resolutions are \(q_\mrm{C}=12\) (smooth, \(625\) parameters) and
    \(q_\mrm{C}=28\) (rough, \(3\,249\) parameters). DART parameters:
    \(\gamma=5\cdot10^{-3}\), \(\theta=0.9\), \(n=6\) for both.}
  \label{fig:hierarchical}
\end{figure}

    \section{Proofs}
        \label{sec:proof}
        The remainder of this manuscript assembles the proof of
Theorem~\ref{thm:mixingTime}. The argument has two stages. In the
first, we analyse an \emph{idealised} chain \(\ideal{T}\) whose
proposals are drawn exactly from the localised surrogate measures
\(\Pi_x\). This chain is \(\Pi\)-reversible, so its mixing time can be
bounded by the \(s\)-conductance machinery. In the second stage, we
transfer the bound to the \emph{implemented} chain \(T\), whose
proposals come from a finite-length root chain.
Section~\ref{sec:proof:main} combines all of these ingredients to
conclude. Standard, cited, and technical auxiliary results are collected
in Appendix~\ref{app:aux}.

The first stage rests on the following 
result, due to \cite{dwivedi2019log}, which we state in notation
adapted to our setting.

\begin{proposition} \label{prp:conductanceLowerBound}
    Let \(t > 0\), \(\varepsilon \in (0, 1)\), and let \(\nu\)
    be a probability measure on \(\R^d\) with density
    \(\varphi \propto e^{-g}\), where \(g \colon \R^d \to \R\) is
    \(\lambda\)-strongly convex.
    Further, suppose \(\{\mc{T}_x\}_{x \in \R^d}\) are the transition
    measures of a Markov chain with invariant measure \(\nu\).
    If \(\mc{K} \subset \R^d\) is a convex set such that
    for any \(x, y \in \mc{K}\)
    \begin{equation}
        \tvnorm{\mc{T}_x - \mc{T}_y} \le \varepsilon
            \quad \text{whenever} \quad
            \enorm{x - y} \le t,
    \end{equation}
    then for any measurable partition of \(\R^d\) into \(A_1\)
    and \(A_2\),
    \begin{equation}    \label{eq:generalConductanceBound}
        \int_{A_1} \mc{T}_z(A_2)\, \varphi(z) \d{z}
        \ge \frac{1 - \varepsilon}{4}
        \min\!\left\{
            1,\,
            \frac{\log 2}{8}\,\nu^2(\mc{K})\, t \sqrt{\lambda}
        \right\}
        \min\!\big\{
            \nu(A_1 \cap \mc{K}),\,
            \nu(A_2 \cap \mc{K})
        \big\}.
    \end{equation}
\end{proposition}

If, in addition, \(\nu\) assigns a significant fraction of its mass to
\(\mc{K}\), Proposition~\ref{prp:conductanceLowerBound} yields an
explicit conductance lower bound, which via \eqref{eq:lovaszBound}
gives a mixing-time upper bound under a warm-start assumption.

\begin{lemma} \label{lem:generalMixingBound}
    Let \(\nu \in \mc{P}(\R^d)\) have density
    \(\varphi \propto e^{-g}\) with \(g \colon \R^d \to \R\)
    \(\lambda\)-strongly convex, and let
    \(\{\mc{T}_x\}_{x \in \R^d}\) be the transition measures of a
    \(\nu\)-reversible, lazy Markov chain with invariant measure
    \(\nu\). Suppose there exist \(t > 0\) and a convex set
    \(\mc{K} \subset \R^d\) with
    \begin{equation} \label{eq:mixingContinuity}
        \tvnorm{\mc{T}_x - \mc{T}_y} \le \tfrac{3}{4}
        \quad \text{whenever} \quad
        \enorm{x - y} \le t, \qquad x, y \in \mc{K}.
    \end{equation}
    Let \(\delta \in (0,1)\) and let \(\mu\) be \(\beta\)-warm with
    respect to \(\nu\), with \(\beta \ge 1\). If
    \begin{equation} \label{eq:targetMassAssumption}
        \nu(\mc{K}) \ge 1 - \frac{\delta}{2\beta},
    \end{equation}
    then
    \begin{equation} \label{eq:generalMixingForAllN}
        \tvnorm{T^n \mu - \nu} \le \delta
        \qquad \text{for all} \qquad
        n \gtrsim \frac{1}{t^2 \lambda} \log \frac{2\beta}{\delta}.
    \end{equation}
\end{lemma}
\begin{proof}
    Set \(s \coloneqq \tfrac{\delta}{2\beta}\). Since \(\beta \ge 1\)
    and \(\delta < 1\), we have \(s \in (0, \tfrac{1}{2})\), and the
    error budget in \eqref{eq:lovaszBound} is split evenly.
    With \(\varepsilon = 3/4\), and taking the minimum in
    \eqref{eq:generalConductanceBound} to be attained by its second
    argument (the complementary case only strengthens the bound),
    \eqref{eq:generalConductanceBound} reduces to
    \begin{equation} \label{eq:reducedConductanceBound}
        \int_{A_1} \mc{T}_z(A_2)\,\varphi(z)\,\d{z}
            \ge \frac{\log 2}{128}\, t \sqrt{\lambda}\,
                \nu^2(\mc{K})\,
                \min\!\big\{
                    \nu(A_1 \cap \mc{K}),\,
                    \nu(A_2 \cap \mc{K})
                \big\}.
    \end{equation}
    Now let \(A\) be a measurable set with
    \(s < \nu(A) \le \tfrac{1}{2}\), and set \(A_1 = A\),
    \(A_2 = \cmpl{A}\). By \eqref{eq:targetMassAssumption},
    \begin{equation} \label{eq:nuCapBound}
        \nu(A \cap \mc{K})
            \ge \nu(A) - \nu(\cmpl{\mc{K}})
            \ge \nu(A) - s,
    \end{equation}
    and analogously \(\nu(\cmpl{A} \cap \mc{K}) \ge \nu(\cmpl{A}) - s\).
    Since \(\nu(A) > s\) and \(\nu(\cmpl{A}) \ge \nu(A) > s\),
    both lower bounds are positive, so
    \[
        \min\!\big\{
            \nu(A \cap \mc{K}),\, \nu(\cmpl{A} \cap \mc{K})
        \big\}
        = \nu(A) - s.
    \]
    Using \(\log 2 \ge 1/2\), \(\nu(\mc{K}) \ge 1 - s\) and
    \((1-s)^2 \ge \tfrac{1}{4}\), we arrive at the conductance lower
    bound
    \begin{equation}
        \Phi_s(T)
        \;\coloneqq\;
        \inf_{\nu(A) \in (s,\, \nicefrac{1}{2})}
        \frac{
            \int_{A} \mc{T}_x(\cmpl{A})\,\varphi(x)\,\d{x}
        }{
            \nu(A) - s
        }
        \;\ge\; \frac{t\sqrt{\lambda}}{1024}.
    \end{equation}
    Inserting this into \eqref{eq:lovaszBound} and applying
    Bernoulli's inequality yields
    \begin{equation} \label{eq:mixingPenultimate}
        \tvnorm{T^n \mu - \nu}
            \le \frac{\delta}{2}
                + \beta\bigl(1 - c\,t^2\lambda\bigr)^n
            \le \frac{\delta}{2}
                + \beta\,e^{-c\,n t^2 \lambda},
    \end{equation}
    where \(c = 1/(2 \cdot 1024^2)\) is a universal constant. For
    \(n \ge c^{-1}(t^2\lambda)^{-1}\log(2\beta/\delta)\) the second term
    is at most \(\delta/2\), giving
    \(\tvnorm{T^n\mu - \nu} \le \delta\) for all such \(n\). The factor
    \(c^{-1}\) is absorbed into the implicit constant of the statement.
\end{proof}

The remaining ingredient is a convex set \(\mc{K}\) satisfying
\eqref{eq:targetMassAssumption}. For strongly log-concave targets, this
is supplied by concentration of \(\nu\) around the minimiser of its
potential \cite{durmus2019high, dwivedi2019log}. The radius of concentration
is governed by
\begin{equation} \label{eq:rDef}
    r \colon (0,\infty) \times \N \to \R, \qquad
    r(\varepsilon, d) \coloneqq \left(
        1 + 2\sqrt{\tfrac{\log \varepsilon^{-1}}{d}}
            + \tfrac{2\log \varepsilon^{-1}}{d}
    \right)^{\!\frac{1}{2}},
\end{equation}
and the concentration of measure is characterised as follows.

\begin{proposition} \label{prp:concentration}
    Let \(\nu\) be a probability measure on \(\R^d\) with density
    \(\varphi \propto e^{-g}\), where \(g \colon \R^d \to \R\) is
    \(\lambda\)-strongly convex, and set
    \(\hat{x} \coloneqq \argmin{\R^d} g\).
    Then for any \(\varepsilon \in (0,1)\),
    \begin{equation}
        \nu\!\left(\mc{K}(\varepsilon)\right) \ge 1 - \varepsilon,
        \qquad \text{where} \qquad
        \mc{K}(\varepsilon)
            \coloneqq \mbb{B}\!\left(
                \hat{x},\, r(\varepsilon, d)\sqrt{d/\lambda}
            \right).
    \end{equation}
\end{proposition}

        \subsection{Properties of Localised Surrogate Measures}
            \label{sec:proof:surrogate}
            We first establish Lipschitz continuity of the map
\(x \mapsto \Pi_x\) in total variation, a key ingredient
required by Proposition~\ref{prp:conductanceLowerBound}.
\begin{lemma}   \label{lem:surrogateContinuity}
    Let \(\{\pi_x\}_{x \in \R^d}\) be a collection of 
    strongly log-concave surrogate densities
    (Definition~\ref{def:stronglyLogConcave}). Then
    \begin{equation}
        \tvnorm[\big]{\Pi_x - \Pi_y}
            \le \frac{\gamma}{2\sqrt{\gamma + \theta m}}\;\enorm{x - y}.
    \end{equation} 
\end{lemma}
\begin{proof}
    Pinsker's inequality gives
    \begin{equation}    \label{eq:pinsker}
        \tvnorm[\big]{\Pi_x - \Pi_y}^2
            \le \tfrac{1}{2}\klDiv{\Pi_x}{\Pi_y},
    \end{equation}
    and the identity
    \begin{equation} \label{eq:klSplit}
        \klDiv{\Pi_x}{\Pi_y}
            = \log \frac{N_y}{N_x}
                + \exv_{z \sim \Pi_x}\big(V_y(z) - V_x(z)\big)
    \end{equation}
    splits the divergence into a normalisation ratio and a potential discrepancy.
    As per \eqref{eq:normConstRewrite}, the
    ratio of normalisation constants can be rewritten as
    \begin{equation}    \label{eq:normConstRewriteRestatement}
        \frac{N_y}{N_x}
            = \exv_{z \sim \Pi_x} e^{-(V_y(z) - V_x(z))}.
    \end{equation}
    Plugging \eqref{eq:normConstRewriteRestatement} into \eqref{eq:pinsker},
    \begin{equation}    \label{eq:klIntermediate01}
        \klDiv{\Pi_x}{\Pi_y}
            = \log \exv_{z \sim \Pi_x} e^{-(V_y(z) - V_x(z))}
                + \exv_{z \sim \Pi_x} \big(V_y(z) - V_x(z)\big).
    \end{equation}
    By definition of the surrogate potentials
    \eqref{eq:surrogatePotentialDef},
    \begin{equation}    \label{eq:VDiffSameArgument}
        V_y(z) - V_x(z)
            = \frac{\gamma}{2} \big(\enorm{z - y}^2 - \enorm{z - x}^2\big) 
            = \gamma \iprod{x - y, z} - c_{x,y},
    \end{equation}
    where
    \(
        c_{x, y} \coloneqq \tfrac{\gamma}{2}
            \big(\enorm{x}^2 - \enorm{y}^2\big).
    \)
    Inserting \eqref{eq:VDiffSameArgument} into
    \eqref{eq:klIntermediate01}, the constants cancel, leaving
    \begin{equation}   \label{eq:klIntermediate02}
        \klDiv{\Pi_x}{\Pi_y}
            = \log \exv_{z \sim \Pi_x}
                e^{-\gamma \iprod{x - y,\;z}}
                + \gamma \iprod{x - y,\;\bar{u}_x}.
    \end{equation}
    Writing \(z = \bar{u}_x + (z - \bar{u}_x)\) inside
    the logarithm, the deterministic contribution
    \(-\gamma\iprod{x - y,~\bar{u}_x}\) exits the
    expectation and cancels with the second term, giving
    \begin{equation}    \label{eq:klMeanFactored}
        \klDiv{\Pi_x}{\Pi_y}
            = \log \exv_{z \sim \Pi_x}
                e^{-\gamma \iprod{x - y,\;z - \bar{z}_x}}.
    \end{equation}
    Since \(g\) is \(m\)-strongly convex, we may write
    \(g = \tfrac{m}{2}\enorm{\cdot}^2 + \varphi\) with
    \(\varphi\) convex.
    Expanding the surrogate potential and completing the square
    in~\(z\) shows that
    \begin{equation}    \label{eq:surrogateDecomposition}
        \pi_x(z)
            \propto e^{-\theta \varphi (z)}
            e^{
                -\frac{\gamma + \theta m}{2}
                    \enorm{z - c_x}^2
            },
    \end{equation}
    where \(c_x \coloneqq \gamma x / (\gamma + \theta m)\).
    The second factor is proportional to the density of
    \(\Gamma_x \coloneqq
        \mc{N}\!\big(c_x,\, (\theta m + \gamma)^{-1} \mbb{I}\big)\),
    and the first is log-concave.
    Harg\'{e}'s inequality
    \cite[Theorem~1.1]{harge2004convex} therefore gives,
    for any convex \(\psi\),
    \begin{equation}    \label{eq:hargeApplication}
        \exv_{z \sim \Pi_x} \psi(z - \bar{z}_x)
            \le \exv_{w \sim \Gamma_x} \psi(w - c_x).
    \end{equation}
    Applying \eqref{eq:hargeApplication} with the convex function
    \(\psi \colon u \mapsto e^{-\gamma \iprod{x - y,\; u}}\)
    and evaluating the resulting Gaussian moment generating function
    yields
    \begin{equation}    \label{eq:hargeBound01}
        \exv_{z \sim \Pi_x}
            e^{-\gamma \iprod{x - y,\;z - \bar{z}_x}}
        \le \exv_{w \sim \mc{N}(0,\, (\theta m + \gamma)^{-1} \mbb{I})}
            e^{-\gamma \iprod{x - y,\;w}}
        = e^{\frac{\gamma^2}{2(\theta m + \gamma)}
            \,\enorm{x - y}^2}.
    \end{equation}
    Inserting \eqref{eq:hargeBound01} into
    \eqref{eq:klMeanFactored} and the result into
    \eqref{eq:pinsker}, and taking the square root yields
    the claim.
\end{proof}

Secondly, we show that the mode displacement of
the localised surrogate measure is entirely controlled by
the regularisation and tempering parameters. Denote the modes
\(a_x \coloneqq \argmin{\R^d} V_x\).

\begin{lemma}   \label{lem:meanModeDistance}
    Let \(\{\pi_x\}_{x \in \R^d}\) be a collection of
    strongly log-concave surrogate densities
    (Definition~\ref{def:stronglyLogConcave}).
    Then
    \begin{equation}    \label{eq:meanModeDistance}
        \enorm{x - a_x}
            \le \frac{\theta M}{\gamma + \theta m}
                \;\enorm{x - \hat{x}}.
    \end{equation}
\end{lemma}
\begin{proof}
    First-order optimality for \(a_x\) gives
    \(\nabla V_x(a_x) = \theta \nabla g(a_x)
        + \gamma (a_x - x) = 0\),
    which yields
    \begin{equation}    \label{eq:optCondition}
        \gamma (x - a_x) = \theta \, \nabla g(a_x).
    \end{equation}
    Since \(\nabla g(\hat{x}) = 0\) and \(g\) is \(M\)-smooth,
    \begin{equation}    \label{eq:gradBound}
        \enorm{\nabla g(a_x)}
            = \enorm{\nabla g(a_x) - \nabla g(\hat{x})}
            \le M \enorm{a_x - \hat{x}}.
    \end{equation}
    Combining \eqref{eq:optCondition} and \eqref{eq:gradBound},
    \begin{equation}    \label{eq:simpleMeanMode}
        \enorm{x - a_x}
            \le \frac{\theta M}{\gamma}\;\enorm{a_x - \hat{x}}.
    \end{equation}
    On the other hand,
    \(\nabla V_x(\hat{x})
        = \theta \nabla g(\hat{x}) + \gamma(\hat{x} - x)
        = \gamma(\hat{x} - x)\).
    Since \(V_x\) is \((\gamma + \theta m)\)-strongly convex
    with minimiser \(a_x\),
    \begin{equation}    \label{eq:scBound}
        \enorm{a_x - \hat{x}}
            \le \frac{1}{\gamma + \theta m}
                \;\enorm{\nabla V_x(\hat{x})}
            = \frac{\gamma}{\gamma + \theta m}
                \;\enorm{x - \hat{x}}.
    \end{equation}
    Inserting \eqref{eq:scBound} into \eqref{eq:simpleMeanMode}
    yields the claim.
\end{proof}

Finally, we record a warmness estimate for a Gaussian 
measure centred at \(x \in \R^d\), with respect to the 
corresponding surrogate measure \(\Pi_x\).
The mode displacement from Lemma~\ref{lem:meanModeDistance} governs
the resulting warmness constant.

\begin{lemma}   \label{lem:surrogateWarmness}
    Let \(\{\pi_x\}_{x \in \R^d}\) be a collection of 
    strongly log-concave surrogate densities
    (Definition~\ref{def:stronglyLogConcave}), and write
    \(\kappa_V \coloneqq (\gamma + \theta M)/(\gamma + \theta m)\).
    Then for all \(x \in \R^d\), the Gaussian measure
    \(\mc{N}\big(x,\, (2(\gamma + \theta M))^{-1} \mbb{I}\big)\)
    is \(\beta\)-warm with respect to \(\Pi_x\), where
    \begin{equation} \label{eq:warm:beta}
        \log \beta
            \le \frac{d}{2}\,\log(2\kappa_V)
                + \frac{\kappa_V \, \theta^2 M^2}{\gamma}
                    \;\enorm{x - \hat{x}}^2.
    \end{equation}
\end{lemma}
\begin{proof}
    The surrogate potential \(V_x\) is
    \((\gamma + \theta m)\)-strongly convex and
    \((\gamma + \theta M)\)-smooth with minimiser~\(a_x\).
    In \cite[Section~3.2.1]{dwivedi2019log}, the authors
    derive bounds on the density ratio for a Gaussian
    whose mean is displaced by \(\varepsilon > 0\) from the
    target mode, where the target density has a
    \(m_{\star}\)-strongly convex and
    \(L_{\star}\)-smooth potential.
    Applying their result with
    \(
        \varepsilon = \enorm{x - a_x}
    \),
    \(
        m_{\star} = \gamma + \theta m
    \)
    and 
    \(
        L_{\star} = \gamma + \theta M
    \)
    yields
    \begin{equation}    \label{eq:warmIntermediate}
        \log \beta
            \le \frac{d}{2}\,
                \log \frac{2(\gamma + \theta M)}{\gamma + \theta m}
                + (\gamma + \theta M)\;\enorm{x - a_x}^2.
    \end{equation}
    The first term is already \(\tfrac{d}{2}\log(2\kappa_V)\) by construction.
    For the second term, Lemma~\ref{lem:meanModeDistance} bounds the
    mode displacement by
    \(\enorm{x - a_x} \le \tfrac{\theta M}{\gamma + \theta m}\,\enorm{x - \hat x}\),
    so that
    \begin{equation}    \label{eq:warmSecondTerm}
        (\gamma + \theta M)\;\enorm{x - a_x}^2
            \le \frac{(\gamma + \theta M)\,\theta^2 M^2}{(\gamma + \theta m)^2}
                \;\enorm{x - \hat{x}}^2
            = \frac{\kappa_V\,\theta^2 M^2}{\gamma + \theta m}
                \;\enorm{x - \hat{x}}^2
            \le \frac{\kappa_V\,\theta^2 M^2}{\gamma}
                \;\enorm{x - \hat{x}}^2,
    \end{equation}
    where the last step uses \(\gamma + \theta m \ge \gamma\).
    This yields~\eqref{eq:warm:beta} and completes the proof.
\end{proof}
            
        \subsection{Bounding the Acceptance Probability}
            \label{sec:proof:acceptance}
            The acceptance probability of \CHAIN{} is governed by the
target potential difference \(f(z) - f(x)\), the surrogate
correction \(\theta(g(x) - g(z))\), and the normalisation constant
ratio \(\log(N_z / N_x)\).
We begin with a bound on the normalisation constant ratio.
Bounds on the potential differences are folded into the proof
of Lemma~\ref{lem:temperedAcceptance}.

\begin{lemma}   \label{lem:normConstRatioBound}
    Let \(\gamma > 0\) and let \(\varphi: \R^d \to \R\) be convex
    and continuously differentiable.
    For each \(x \in \R^d\), define the potential
    \begin{equation}
        V_x: \R^d \to \R, \quad
            y \mapsto \varphi(y) + \frac{\gamma}{2}~\enorm{y - x}^2.
    \end{equation}
    Let \(\Pi_x\) be the probability measure on
    \((\R^d, \mf{B}(\R^d))\) with density
    \(\pi_x \propto e^{-V_x}\) and normalisation
    constant \(N_x\). Then, for all \(x, y \in \R^d\),
    \begin{equation}
        \log \frac{N_y}{N_x}
            \le \iprod[\big]{
                \exv_{u \sim \Pi_x} \! \nabla \varphi(u),~
                x-y
            }.
    \end{equation}
\end{lemma}
\begin{proof}
    As in~\eqref{eq:normConstRewrite}, we express the
    ratio of normalisation constants as an expectation
    over~\(\Pi_x\):
    \begin{equation}    \label{eq:normConstRewrite01}
        \log\frac{N_y}{N_x}
            =  \log \exv_{u \sim \Pi_x} e^{V_x(u) - V_y(u)}.
    \end{equation}
    Writing \(u - y = (u - x) + (x - y)\) and expanding,
    the difference in potentials is
    \begin{equation}    \label{eq:surrogatePotentialDifference}
        V_x(u) - V_y(u)
            = \frac{\gamma}{2}\big(
                \enorm{u-x}^2 - \enorm{u-y}^2
            \big)
            = -\gamma\iprod{x-y,\;u}
                + \gamma\iprod{x-y,\;x}
                - \frac{\gamma}{2}\enorm{x-y}^2.
    \end{equation}
    The constant terms coincide with the negative logarithm of
    the Laplace transform of
    \(\mc{N}(x,\gamma^{-1}\mbb{I})\) evaluated at
    \(\gamma(x-y)\), which allows us to write
    \begin{equation}    \label{eq:normConstRewrite02}
        \log \frac{N_y}{N_x}
            = \log \exv_{u \sim \Pi_x} e^{ -\gamma \iprod{x-y, u} }
                - \log \exv_{ v \sim \mc{N}(x, \gamma^{-1} \mbb{I}) }
                e^{ -\gamma \iprod{x-y, v} }.
    \end{equation}
    We now apply Harg\'{e}'s inequality to bound the first
    term.
    Since \(\varphi\) is convex, \(\pi_x\) is a log-concave
    perturbation of the Gaussian
    \(\mc{N}(x, \gamma^{-1}\mbb{I})\), so
    \cite[Theorem~1.1]{harge2004convex} gives
    \begin{equation}    \label{eq:hargeBound}
        \exv_{u \sim \Pi_x} e^{ -\gamma \iprod{x-y,~u - \bar{u}_x} }
            \le \exv_{ v \sim \mc{N}(x, \gamma^{-1} \mbb{I}) }
            e^{ -\gamma \iprod{x-y,~v-x} }.
    \end{equation}
    Writing \(u = \bar{u}_x + (u - \bar{u}_x)\)
    inside the first term of \eqref{eq:normConstRewrite02}
    and applying \eqref{eq:hargeBound},
    \begin{align}
        \log \exv_{u \sim \Pi_x} e^{ -\gamma \iprod{x-y, u} }
            &= -\gamma \iprod{x-y,~\bar{u}_x}
                + \log \exv_{u \sim \Pi_x}
                    e^{ -\gamma \iprod{x-y,~u-\bar{u}_x} } \\
            &\le -\gamma \iprod{x-y,~\bar{u}_x}
                + \log \exv_{ v \sim \mc{N}(x, \gamma^{-1} \mbb{I}) }
                e^{ -\gamma \iprod{x-y,~v-x} } \\
    \label{eq:hargeBoundImplication}
            &= \gamma \iprod{x-y,~x-\bar{u}_x}
                + \log \exv_{ v \sim \mc{N}(x, \gamma^{-1} \mbb{I}) }
                e^{ -\gamma \iprod{x-y,~v} }.
    \end{align}
    The Laplace transform of the Gaussian density in \eqref{eq:hargeBoundImplication}
    cancels with the second term of
    \eqref{eq:normConstRewrite02}, leaving
    \begin{equation}    \label{eq:logRatioBound02}
        \log \frac{N_y}{N_x}
            \le \gamma \iprod{x-\bar{u}_x,~x-y}.
    \end{equation}
    The proof concludes by substituting the Stein identity
    \eqref{eq:steinIdentity} into \eqref{eq:logRatioBound02}.
\end{proof}

With the normalisation constant ratio under control,
we can bound the expected acceptance probability.

\begin{lemma}\label{lem:temperedAcceptance}
    Let \(f : \R^d \to \R\) be \(\lambda\)-strongly convex
    and \(L\)-smooth with minimiser
    \(x^{\star} \coloneqq \argmin{\R^d} f\),
    and define \(\kappa \coloneqq L / \lambda\).
    Consider an \(m\)-strongly convex 
    and \(M\)-smooth surrogate
    \(g : \R^d \to \R\) satisfying
    \(\lambda \le m \le M \le L\),
    and let \(\hat{x} \coloneqq \argmin{\R^d} g\).
    Fix \(\tau \in (0, 1)\) and parameters
    \(\theta \in (0, 1)\), \(\gamma > 0\),
    and let \(\pi_x\) be the strongly log-concave
    surrogate density (Definition~\ref{def:stronglyLogConcave})
    using \(g\) as the surrogate for~\(f\),
    with normalisation constant~\(N_x\).
    Define
    \(
        \mc{K}_\tau
            \coloneqq \mbb{B}\big(
                x^{\star},~r(\tau,d) \sqrt{d / \lambda}
            \big)
    \)
    and assume that \(g\) approximates
    \(f\) in the sense that
    \begin{center}
    \begin{minipage}{0.8\linewidth}
        \begin{enumerate}[label=(a\arabic*), ref=a\arabic*]
            \item \label{eq:surrogateModeAssmp}
                \(
                    \hat{x} \in \mc{K}_\tau, \quad
                \)
                and
            \item \label{eq:gradientDiffAssmp}
                \(
                    \sup\limits_{x \in \mc{K}_\tau}
                    \! \enorm{\nabla g(x) - \nabla f(x)}
                        \le \eta 
                \) for some \(\eta > 0\).
        \end{enumerate}
    \end{minipage}
    \end{center}
    Define the acceptance probability
    \(
        \alpha(x,y) \coloneqq
        \min\!\big\{
            1,\, e^{-A(x, y)}
        \big\},
    \)
    where
    \begin{equation}
        A(x, y)
            \coloneqq f(y) - f(x)
                + \theta \big (
                    g(x) - g(y)
                \big )
                + \log \frac{N_y}{N_x}.
    \end{equation}
    Then, for any
    \(x \in \mc{K}_\tau\) and \(\varepsilon \in (0, 1)\),
    we have 
    \(
        \exv_{z \sim \Pi_x} \alpha(x,z)
            \ge 1-\varepsilon,
    \)
    provided that the parameter choice satisfies
    \begin{equation}    \label{eq:acceptanceConditions}
        \theta = \frac{1}{2}
    \quad \text{and} \quad
        \gamma
            \ge c(\tau,d) \left(
                    L \kappa
                    + \dfrac{L d}{\varepsilon} 
                    + \dfrac{\eta^2 d}{\varepsilon^2}
            \right),
    \end{equation}
    where \(c(\tau,d) = 4(1+r(\tau,d))^2\).
    
\end{lemma}
\begin{proof}
    Fix \(x \in \mc{K}_\tau\) and write
    \(\sigma \coloneqq (\gamma + \nicefrac{\lambda}{2})^{-\nicefrac{1}{2}}\)
    and \(A_{+}(x, z) \coloneqq \max \{0, \, A(x,z)\}\).
    Since \(\min \{1,~e^{-t}\} \ge 1 - \max\{0, t\}\) for all
    \(t \in \R\),
    it suffices to show
    \(\exv_{z \sim \Pi_x} A_{+}(x,z) \le \varepsilon\).
    Setting \(\theta = 1/2\), \(L\)-smoothness of \(f\) gives
    \(
        f(z) - f(x)
            \le \iprod{-\nabla f(x),\, x - z}
                + \frac{L}{2} \enorm{x - z}^2
    \),
    convexity of \(g\) gives
    \(
        \frac{1}{2}(g(x) - g(z))
            \le \frac{1}{2}\iprod{\nabla g(x),\, x - z}
    \),
    and Lemma~\ref{lem:normConstRatioBound} applied with
    \(\varphi = \frac{1}{2} g\) gives
    \(
        \log(N_z / N_x)
            \le \frac{1}{2}\iprod{
                \exv_{u \sim \Pi_x} \nabla g(u),\;
                x-z
            }
    \).
    Combining these three bounds yields
    \begin{equation}    \label{eq:ABound}
        A(x, z)
            \le \iprod{w,\, x - z}
                + \frac{L}{2} \enorm{x - z}^2,
    \end{equation}
    where, after regrouping the gradient terms,
    \begin{equation}    \label{eq:wExplicit}
        w
            \coloneqq \nabla g(x) - \nabla f(x)
                + \frac{1}{2} \exv_{u \sim \Pi_x} \! \big(
                    \nabla g(u) - \nabla g(x)
                \big).
    \end{equation}
    By~\eqref{eq:gradientDiffAssmp},
    \(M\)-smoothness of~\(g\), and Jensen's inequality,
    \begin{equation}    \label{eq:normWBound}
        \enorm{w}
            \le \eta
                + \tfrac{M}{2} \sqrt{S},
    \end{equation}
    where
    \(
        S \coloneqq \exv_{z \sim \Pi_x} \enorm{z - x}^2.
    \)
    Cauchy--Schwarz followed by Jensen's inequality gives
    \begin{equation}    \label{eq:expectedA}
        \exv_{z \sim \Pi_x} A_{+}(x,z)
            \le \enorm{w}\sqrt{S}
                + \tfrac{L}{2} S
            \le \eta \sqrt{S} + L S,
    \end{equation}
    where the last step
    substitutes~\eqref{eq:normWBound}
    and uses \(M \le L\).
    It remains to bound \(S\) in terms of the algorithm
    parameters.
    The triangle inequality gives
    \(\sqrt{S}
        \le \enorm{x - a_x}
            + \sqrt{\exv \enorm{z - a_x}^2}\).
    Since both \(x\) and \(\hat{x}\) lie in \(\mc{K}_\tau\) by
    \eqref{eq:surrogateModeAssmp}, the triangle inequality yields
    \(\enorm{x - \hat{x}} \le \mrm{diam}(\mc{K}_\tau) = 2 r(\tau,d)\sqrt{d/\lambda}\).
    Combining with Lemma~\ref{lem:meanModeDistance},
    \begin{equation}    \label{eq:modeDistance}
        \enorm{x - a_x}
            \le \frac{\theta M}{\gamma + \theta m}
                \;\enorm{x - \hat{x}}
            \le \frac{L}{2}\,\sigma^2
                    \cdot 2 r(\tau,d) \sqrt{d/\lambda}
            = r(\tau,d) L \sqrt{d/\lambda}\;\sigma^2,
    \end{equation}
    where the second inequality uses \(\theta M \le L/2\),
    \(M \le L\), and \((\gamma + m/2)^{-1} \le \sigma^2\)
    (the last since \(m \ge \lambda\)).
    From 
    \cite[Proposition~1(ii)]{durmus2019high},
    \(
        \exv \enorm{z - a_x}^2
        \le \sigma^2 d,
    \)
    so
    \begin{equation}    \label{eq:sqrtSBound}
        \sqrt{S}
            \le r(\tau,d) L \sqrt{d / \lambda}
                \; \sigma^2
            + \sigma \sqrt{d}
            = \sigma \sqrt{d}
                \Big(
                    1 + r(\tau,d) \sqrt{L \kappa}
                        \; \sigma
                \Big).
    \end{equation}
    Write \(c_s \coloneqq 1 + r(\tau,d)\).
    We first impose \(\gamma \ge L\kappa\). Since
    \(\sigma^{-2} = \gamma + \lambda/2 \ge \gamma\), this gives
    \(L\kappa\,\sigma^2 \le 1\), i.e.\ \(\sqrt{L\kappa}\,\sigma \le 1\),
    so inequality~\eqref{eq:sqrtSBound} simplifies to
    \(\sqrt{S} \le c_s\,\sigma\sqrt{d}\). Substitution
    into~\eqref{eq:expectedA} gives
    \begin{equation}
        \exv_{z\sim\Pi_x} A_{+}(x,z)
            \le c_s\,\eta\,\sigma\sqrt{d}
                + c_s^2\,L\,\sigma^2 d.
    \end{equation}
    Bounding each of the two summands by \(\varepsilon/2\), and again
    using \(\sigma^{-2} = \gamma + \lambda/2 \ge \gamma\), yields
    \begin{equation}    \label{eq:twoRequirements}
        \gamma \ge 2 c_s^2 L d / \varepsilon
        \qquad \text{and} \qquad
        \gamma \ge 4 c_s^2 \eta^2 d / \varepsilon^2.
    \end{equation}
    Together with \(\gamma \ge L\kappa\), all three conditions are
    implied by~\eqref{eq:acceptanceConditions}, since
    \(c(\tau,d) \ge 1\), \(c(\tau,d) \ge 2c_s^2\) and \(c(\tau,d) \ge 4c_s^2\).
    Therefore \(\exv_{z\sim\Pi_x}\alpha(x,z) \ge 1-\varepsilon\).
\end{proof}
            
        \subsection{Proof of Theorem~\ref{thm:mixingTime}}
            \label{sec:proof:main}
            Throughout, write \(\mc{T}_x\) for the MH
transition measure that proposes from
\(\mc{P}_x \coloneqq Q_x^{n}\mu_x\) and accepts with probability
\(\alpha\) as in \eqref{eq:dartAcceptance} and \(\ideal{\mc{T}}_x\) for
the transition measure that proposes from \(\Pi_x\) and accepts
with this same probability. Since \(\alpha\) is the
exact MH ratio for proposals drawn from
\(\Pi_x\), the kernel \(\ideal{\mc{T}}_x\) satisfies detailed
balance with respect to \(\Pi\), so the chain is \(\Pi\)-reversible.
The proof proceeds in three stages. We first bound the mixing time of
\(\ideal{T}\), then transfer the bound to \(T\) via the
perturbation estimate of Lemma~\ref{lem:chainPerturbation}, and
finally, verify that condition~\eqref{eq:surrogateStepsCondition} supplies
the required root chain accuracy.

Define
\(\mc{K}_\tau \coloneqq \mbb{B}\big(x^\star,\, r(\delta/(4\beta), d)\sqrt{d/\lambda}\big)\).
Under the constraint \(\log\tfrac{2\beta}{\delta} \le d\) we have
\(\log\tfrac{4\beta}{\delta}/d \le 1 + \tfrac{\log 2}{d} \le 1 + \log 2\),
so
\begin{equation}    \label{eq:uniformRBound}
    r(\delta/(4\beta), d)^2
        \le 1 + 2\sqrt{1 + \log 2} + 2(1 + \log 2)
        \le 9,
\end{equation}
giving \(r(\delta/(4\beta), d) \le 3\) and hence the inclusion
\(\mc{K}_\tau \subset \mc{K}\). Consequently,
Proposition~\ref{prp:concentration} implies
\begin{equation}    \label{eq:targetMassSatisfied}
    \Pi(\mc{K})
        \ge \Pi(\mc{K}_\tau)
        \ge 1 - \frac{\delta}{4 \beta}.
\end{equation}
We next establish the TV-continuity of the measures
\(\ideal{\mc{T}}_x\) on \(\mc{K}\). Since \(\ideal{\mc{T}}_x\)
corresponds to MH proposals from \(\Pi_x\), the set
\(\{x\}\) is an atom of \(\ideal{\mc{T}}_x\) with mass
\(1 - \exv_{z \sim \Pi_x}\alpha(x,z)\), while on
\(\R^d \setminus \{x\}\) the transition measure admits the density
\(\alpha(x, \cdot)\,\pi_x \le \pi_x\). A direct computation gives
\begin{equation}    \label{eq:idealisedTvToExv}
    \tvnorm{\ideal{\mc{T}}_x - \Pi_x}
        = 1 - \exv_{z \sim \Pi_x}\alpha(x, z).
\end{equation}
By the triangle inequality, for \(y_1, y_2 \in \mc{K}\),
\begin{equation}    \label{eq:idealisedDecomposition}
    \tvnorm{\ideal{\mc{T}}_{y_1} - \ideal{\mc{T}}_{y_2}}
        \le \underbrace{
            2\Big(1 - \inf_{x \in \mc{K}}
                \exv_{z \sim \Pi_x}\alpha(x,z)\Big)
        }_{I}
        + \underbrace{
            \tvnorm{\Pi_{y_1} - \Pi_{y_2}}
        }_{II}.
\end{equation}
It suffices to show that each term is at most \(1/4\), so that
\(\tvnorm{\ideal{\mc{T}}_{y_1} - \ideal{\mc{T}}_{y_2}} \le 1/2\).

For Term~\(I\), choose \(\tau^\star \in (0,1)\) so that
\(r(\tau^\star, d) = 3\), whence \(\mc{K}_{\tau^\star} = \mc{K}\)
and \(c(\tau^\star, d) = 4(1 + 3)^2 = 64\). The gradient fidelity
condition~\eqref{eq:fidelityCondition} supplies
\(\eta^2 \lesssim L\max\{1, \kappa/d\}\) in
assumption~\eqref{eq:gradientDiffAssmp}, and the localisation
condition~\eqref{eq:localisationCondition} gives
\(\theta = 1/2\) and \(\gamma \gtrsim L\max\{\kappa, d\}\). Applying
Lemma~\ref{lem:temperedAcceptance} with \(\varepsilon = 1/8\), the
requirement~\eqref{eq:acceptanceConditions} reads
\(\gamma \ge 64\big(L\kappa + 8 L d + 64\,\eta^2 d\big)\), and since
\(\eta^2 d \lesssim L\max\{d, \kappa\}\), each summand is bounded by
a multiple of \(L\max\{\kappa, d\}\). Thus
\eqref{eq:localisationCondition} implies the requirement. We obtain
\begin{equation}    \label{eq:acceptanceLowerBound}
    \inf_{x \in \mc{K}}
        \exv_{z \sim \Pi_x}\alpha(x,z) \ge \frac{7}{8},
\end{equation}
and therefore \(I \le 2(1 - 7/8) = 1/4\).

For Term~\(II\), Lemma~\ref{lem:surrogateContinuity} gives
\begin{equation}    \label{eq:surrogateLipschitz}
    \tvnorm{\Pi_{y_1} - \Pi_{y_2}}
        \le \frac{\gamma}{2\sqrt{\gamma + \theta m}}
            \;\enorm{y_1 - y_2}.
\end{equation}
Since \(\gamma \gtrsim L \ge m = 2\theta m\), the Lipschitz constant
satisfies \(\frac{\gamma}{2\sqrt{\gamma + \theta m}} \lesssim \sqrt{\gamma}\),
so \(II \le 1/4\) whenever
\begin{equation}    \label{eq:contractionDistance}
    \enorm{y_1 - y_2}
        \lesssim \gamma^{-1/2}
        \eqqcolon t.
\end{equation}
Collecting the two bounds establishes
\(\tvnorm{\ideal{\mc{T}}_{y_1} - \ideal{\mc{T}}_{y_2}} \le 1/2\)
whenever \(\enorm{y_1 - y_2} \le t\) with \(t \eqsim \gamma^{-1/2}\),
for all \(y_1, y_2 \in \mc{K}\). A lazy version of \(\ideal{T}\)
(c.f.~Appendix~\ref{app:aux}) satisfies
continuity with constant \(3/4\), which is the
hypothesis~\eqref{eq:mixingContinuity} of
Lemma~\ref{lem:generalMixingBound}. Since \(\ideal{T}\) is
\(\Pi\)-reversible, lazy, and \eqref{eq:targetMassSatisfied} supplies
the mass requirement~\eqref{eq:targetMassAssumption} at level
\(\delta/2\), Lemma~\ref{lem:generalMixingBound} applies to
\(\ideal{T}\) with tolerance \(\delta/2\), giving
\begin{equation}    \label{eq:idealisedMixing}
    \tvnorm{\ideal{T}^N \mu - \Pi} \le \frac{\delta}{2}
    \quad \text{for all} \quad
    N \ge N_\delta
        \coloneqq \Big\lceil
            C\, t^{-2}\lambda^{-1}\log\tfrac{4\beta}{\delta}
        \Big\rceil,
\end{equation}
where \(C > 0\) is the universal constant of
Lemma~\ref{lem:generalMixingBound}. Since \(t \eqsim \gamma^{-1/2}\)
and \(\gamma \eqsim L\max\{\kappa, d\}\),
\begin{equation}    \label{eq:NdeltaSize}
    N_\delta
        \eqsim \frac{\gamma}{\lambda}\log\frac{4\beta}{\delta}
        \eqsim \kappa \max\{\kappa, d\}\log\frac{2\beta}{\delta}.
\end{equation}

Both chains, \(\ideal{T}\) and \(T\), accept with the common probability
\(\alpha(x, \cdot)\). Hence, for each measurable \(A\), \(\mc{T}_x(A)\) is the integral of
the \([0,1]\)-valued function
\(z \mapsto \alpha(x, z)\,\delta_z(A) + (1 - \alpha(x, z))\,\delta_x(A)\)
against \(\mc{P}_x\), and \(\ideal{\mc{T}}_x(A)\) the integral of the
same function against \(\Pi_x\). The functional bound
\eqref{eq:tvFunctionalBound} and a supremum over \(A\) give
\begin{equation}    \label{eq:kernelGap}
    \tvnorm{\mc{T}_x - \ideal{\mc{T}}_x}
        \le \tvnorm{\mc{P}_x - \Pi_x}.
\end{equation}
Define the enlarged region
\(\mc{K}_N \coloneqq \mbb{B}\big(x^\star,\,
    r(\delta/(4\beta N_\delta), d)\sqrt{d/\lambda}\big)\).
Proposition~\ref{prp:concentration} gives
\(\Pi(\cmpl{\mc{K}_N}) \le \delta/(4\beta N_\delta)\). Applying the
perturbation estimate of Lemma~\ref{lem:chainPerturbation} to the
lazy chains, whose per-step difference halves that of the kernels
\(\mc{T}_x, \ideal{\mc{T}}_x\) (Appendix~\ref{app:aux}), with
\(X = \mc{K}_N\), \(\nu = \Pi\), and \(N = N_\delta\), and
substituting \eqref{eq:kernelGap},
\begin{equation}    \label{eq:perturbationSplit}
    \tvnorm{T^{N_\delta}\mu - \ideal{T}^{N_\delta}\mu}
        \le N_\delta \sup_{x \in \mc{K}_N}
            \tfrac12\,\tvnorm{\mc{T}_x - \ideal{\mc{T}}_x}
        + \beta N_\delta\,\Pi(\cmpl{\mc{K}_N})
        \le \frac{N_\delta}{2} \sup_{x \in \mc{K}_N}
            \tvnorm{\mc{P}_x - \Pi_x}
        + \frac{\delta}{4}.
\end{equation}
It remains to make the first term at most \(\delta/4\).
Under \eqref{eq:localisationCondition}, the condition number of each
surrogate density \(\pi_x\) is uniformly bounded:
\begin{equation}    \label{eq:surrogateConditionNumber}
    \kappa_x
        \coloneqq
        \frac{\gamma + \theta M}{\gamma + \theta m}
        \le 1 + \frac{M}{2\gamma}
        \lesssim 1,
\end{equation}
since \(\gamma \gtrsim L \ge M\) and in particular
\(\tilde{\kappa} \coloneqq \sup_x \kappa_x \lesssim 1\).
For \(x \in \mc{K}_N\), the triangle inequality through \(x^\star\),
using \(x \in \mc{K}_N\) and the standing assumption \(\hat x \in \mc{K}\),
gives
\begin{equation}    \label{eq:KNdiameter}
    \enorm{x - \hat x}^2
        \le \big(
            r(\delta/(4\beta N_\delta), d) + 3
        \big)^2 \frac{d}{\lambda}
        \lesssim \frac{d + \log(2 N_\delta)}{\lambda},
\end{equation}
where the last step uses
\(r(\delta/(4\beta N_\delta), d)^2 \lesssim 1 + \log(4\beta N_\delta/\delta)/d\)
and \(\log(4\beta/\delta) \le 2d\) (valid for all \(d \ge 1\) under
\(\log(2\beta/\delta) \le d\)). Lemma~\ref{lem:surrogateWarmness}
then gives, with \(\kappa_x \lesssim 1\) and
\(\theta^2 M^2/\gamma \le L^2/(4\gamma) \lesssim L/\max\{\kappa, d\}\)
from~\eqref{eq:localisationCondition},
\begin{equation}    \label{eq:warmnessBound}
    \log \beta_x
        \le \frac{d}{2}\log(2\kappa_x)
            + \frac{\kappa_x\,\theta^2 M^2}{\gamma}\,\enorm{x - \hat x}^2
        \lesssim d
            + \frac{L}{\max\{\kappa, d\}}\cdot
                \frac{d + \log(2N_\delta)}{\lambda}.
\end{equation}
Since \(L/(\lambda\max\{\kappa, d\}) = \kappa/\max\{\kappa, d\} \le 1\),
the second term is \(\lesssim d + \log(2N_\delta)\), so
\begin{equation}    \label{eq:logWarmness}
    \sup_{x \in \mc{K}_N} \log \beta_x
        \lesssim d + \log(2 N_\delta).
\end{equation}
Now apply Condition~\ref{cnd:root} at each \(x \in \mc{K}_N\), with
the Gaussian initial measure \(\mu_x\) of
Lemma~\ref{lem:surrogateWarmness} and accuracy
\(\varepsilon_{\mathrm{root}} \coloneqq \delta/(2 N_\delta)\). A
common root chain length \(n\) is admissible provided
\begin{equation}    \label{eq:nRequirement}
    n \gtrsim d^{\omega}\tilde{\kappa}^{\tilde{\omega}}
        \sup_{x \in \mc{K}_N}
            \log\frac{2\beta_x}{\varepsilon_{\mathrm{root}}}
        \lesssim d^{\omega}\Big(
            d + \log(2N_\delta) + \log\tfrac{N_\delta}{\delta}
        \Big),
\end{equation}
where we used \(\tilde{\kappa} \lesssim 1\) and
\eqref{eq:logWarmness}. By \eqref{eq:NdeltaSize} and
\(\max\{\kappa, d\} \le \kappa d\), \(\log(2\beta/\delta) \le d\),
we have \(N_\delta \lesssim \kappa^2 d^2\), so
\(\log(2N_\delta) \lesssim 1 + \log\kappa + \log d\) and
\(\log(N_\delta/\delta) \lesssim 1 + \log(\kappa/\delta) + \log d\).
Since \(\delta < 1\) gives \(\log\kappa \le \log(\kappa/\delta)\), and
\(1, \log d \le d\), both terms are \(\lesssim d + \log(\kappa/\delta)\).
Hence the right-hand side
of~\eqref{eq:nRequirement} is \(\lesssim d^{\omega}(d + \log(\kappa/\delta))\),
which is exactly~\eqref{eq:surrogateStepsCondition}. With this \(n\),
\(\sup_{x \in \mc{K}_N}\tvnorm{\mc{P}_x - \Pi_x} \le \delta/(2 N_\delta)\),
so the first term in~\eqref{eq:perturbationSplit} is at most
\(\frac{N_\delta}{2} \cdot \delta/(2 N_\delta) = \delta/4\).

In summary, the first term in~\eqref{eq:perturbationSplit} is at most
\(\delta/4\), so together with its second term,
\(\tvnorm{T^{N_\delta}\mu - \ideal{T}^{N_\delta}\mu} \le \delta/2\).
Combining with \eqref{eq:idealisedMixing} at \(N = N_\delta\),
\begin{equation}
    \tvnorm{T^{N_\delta}\mu - \Pi}
        \le \tvnorm{T^{N_\delta}\mu - \ideal{T}^{N_\delta}\mu}
            + \tvnorm{\ideal{T}^{N_\delta}\mu - \Pi}
        \le \frac{\delta}{2} + \frac{\delta}{2}
        = \delta.
\end{equation}
Hence \(t_\delta(\mu) \le N_\delta \eqsim
\kappa\max\{\kappa, d\}\log(2\beta/\delta)\) by
\eqref{eq:NdeltaSize}, completing the proof.

    \bigskip{\sffamily\textbf{Acknowledgements}}

    This work is supported by the Deutsche Forschungsgemeinschaft (German Research
    Foundation) under Germany’s Excellence Strategy EXC 2181/1 (the Heidelberg
    STRUCTURES Excellence Cluster)

    R.~Kutri would like to thank the Isaac Newton Institute for 
    Mathematical Sciences for the support and hospitality during the programme
    \enquote{Representing, calibrating \& leveraging prediction uncertainty from statistics to machine
    learning}
    where work on this paper was undertaken.
    This work was supported by EPSRC grant no EP/R014604/1.

    \printbibliography

    \appendix

    \section{Technical Auxiliary Results}
        \label{app:aux}
        Laziness is required by the conductance bound \eqref{eq:lovaszBound}.
The \(\tfrac12\)-lazy version of a Markov chain with transition measures
\(\{\mc{T}_x\}_{x \in \R^d}\) has transition measures
\(\tfrac12\delta_x + \tfrac12\mc{T}_x\), which preserves the invariant
measure and reversibility while forcing a non-negative spectrum and
removing periodicity. It rescales the constants of
Section~\ref{sec:proof} in two elementary ways.

Comparing two lazy chains at the same state \(x\), the shared atom
\(\tfrac12\delta_x\) cancels, so
\begin{equation}
    \tvnorm[\big]{
        (\tfrac12\delta_x + \tfrac12\mc{T}_x)
        - (\tfrac12\delta_x + \tfrac12\ideal{\mc{T}}_x)
    }
    = \tfrac12\tvnorm{\mc{T}_x - \ideal{\mc{T}}_x}.
\end{equation}
This is the halving used in \eqref{eq:perturbationSplit}. Comparing a
single lazy chain at two states \(y_1 \ne y_2\), the atoms do not
cancel, and \(\tvnorm{\delta_{y_1} - \delta_{y_2}} \le 1\), so
\begin{equation}
    \tvnorm[\big]{
        (\tfrac12\delta_{y_1} + \tfrac12\ideal{\mc{T}}_{y_1})
        - (\tfrac12\delta_{y_2} + \tfrac12\ideal{\mc{T}}_{y_2})
    }
    \le \tfrac12\tvnorm{\delta_{y_1} - \delta_{y_2}}
        + \tfrac12\tvnorm{\ideal{\mc{T}}_{y_1} - \ideal{\mc{T}}_{y_2}}
    \le \tfrac12 + \tfrac12\cdot\tfrac12
    = \tfrac34.
\end{equation}
A continuity constant of \(1/2\) for the underlying kernels therefore
becomes \(3/4\) for the lazy chain, as used before
Lemma~\ref{lem:generalMixingBound}. See \cite{lovasz1993random} for
further detail.

The transfer from the idealised to the implemented chain rests on the
following perturbation estimate, obtained by a standard telescoping
argument with localisation to \(X\) in the spirit of
\cite{pillai2014ergodicity}.
\begin{lemma} \label{lem:chainPerturbation}
    Let \(T\) and \(\ideal{T}\) be transition operators on
    \(\mc{P}(\R^d)\) with transition measures
    \(\{\mc{T}_x\}_{x \in \R^d}\) and
    \(\{\ideal{\mc{T}}_x\}_{x \in \R^d}\), let \(\nu\) be invariant
    under \(\ideal{T}\), and let \(\mu\) be \(\beta\)-warm with respect
    to \(\nu\), with \(\beta \ge 1\). Then for any measurable
    \(X \subseteq \R^d\) and any \(N \in \N\),
    \begin{equation}
        \tvnorm{T^N \mu - \ideal{T}^N \mu}
            \le N \sup_{x \in X} \tvnorm{\mc{T}_x - \ideal{\mc{T}}_x}
                + \beta N \, \nu(\cmpl{X}).
    \end{equation}
\end{lemma}
\begin{proof}
    Writing \(\nu_k \coloneqq \ideal{T}^k \mu\), define
    \(\mu_k \coloneqq T^{N-k} \nu_k\) for \(k = 0, \dots, N\), so that
    \(\mu_0 = T^N \mu\) and \(\mu_N = \ideal{T}^N \mu\). The triangle
    inequality gives
    \begin{equation} \label{eq:interpolation}
        \tvnorm{T^N \mu - \ideal{T}^N \mu}
            \le \sum_{k=0}^{N-1} \tvnorm{\mu_k - \mu_{k+1}}.
    \end{equation}
    Consecutive interpolants differ only in their \((k+1)\)-st step,
    \(\mu_k = T^{N-1-k}\big(T \nu_k\big)\) and
    \(\mu_{k+1} = T^{N-1-k}\big(\ideal{T} \nu_k\big)\). Both are images
    of \(T \nu_k\) and \(\ideal{T} \nu_k\) under the common operator
    \(T^{N-1-k}\), which is non-expansive in total variation
    by~\eqref{eq:tvFunctionalBound}, so
    \begin{equation}
        \tvnorm{\mu_k - \mu_{k+1}}
            \le \tvnorm{T \nu_k - \ideal{T} \nu_k}.
    \end{equation}
    For the single-step difference and any measurable \(A\),
    \begin{equation}
        \big| (T\nu_k)(A) - (\ideal{T}\nu_k)(A) \big|
            \le \int_{\R^d}
                \big| \mc{T}_x(A) - \ideal{\mc{T}}_x(A) \big|
                \,\d{\nu_k}(x)
            \le \sup_{x \in X} \tvnorm{\mc{T}_x - \ideal{\mc{T}}_x}
                + \nu_k(\cmpl{X}),
    \end{equation}
    splitting the integral over \(X\) and \(\cmpl{X}\), bounding
    \(\nu_k(X) \le 1\) on the first part and
    \(\big|\mc{T}_x(A) - \ideal{\mc{T}}_x(A)\big| \le 1\) on the second.
    Taking the supremum over \(A\),
    \begin{equation}
        \tvnorm{T \nu_k - \ideal{T} \nu_k}
            \le \sup_{x \in X} \tvnorm{\mc{T}_x - \ideal{\mc{T}}_x}
                + \nu_k(\cmpl{X}).
    \end{equation}
    Finally, warmness is preserved under the \(\nu\)-invariant operator
    \(\ideal{T}\). With an analogous argument to \eqref{eq:warmnessContinuityArgument},
    for every measurable \(A\),
    \begin{equation}
        \nu_k(A)
            = \int \ideal{\mc{T}}^{\,k}_y(A)\,\d{\mu}(y)
            \le \beta \int \ideal{\mc{T}}^{\,k}_y(A)\,\d{\nu}(y)
            = \beta\,(\ideal{T}^k \nu)(A)
            = \beta\,\nu(A),
    \end{equation}
    using \(\ideal{T}\)-invariance in the last step, so
    \(\nu_k(\cmpl{X}) \le \beta\, \nu(\cmpl{X})\). Summing the \(N\)
    terms in \eqref{eq:interpolation} yields the claim.
\end{proof}

\end{document}